\documentclass[12pt, reqno]{amsart}
\usepackage{amsmath,amssymb,amsthm,amsxtra, setspace}
\usepackage{amssymb}
\usepackage{mathrsfs}
\usepackage{alltt}
\usepackage{graphicx,type1cm,xcolor}
\usepackage{mathtools, accents}
\usepackage{hyperref}
\usepackage{alltt}
\usepackage{fouridx}
\usepackage{dsfont}
\usepackage{cancel}
\usepackage[margin=1in]{geometry}
\usepackage{shadethm}

\allowdisplaybreaks

\newcommand{\eps}{\varepsilon}
\newcommand{\curl}{\mathrm{curl}}
\newcommand{\supp}{\mathrm{supp}}
\numberwithin{equation}{section}

\newcommand{\what}{\widehat}
\newcommand{\wtilde}{\widetilde}
\newshadetheorem{Theorem}{Theorem}
\newshadetheorem{Proposition}{Proposition}
\newshadetheorem{Lemma}{Lemma}
\newshadetheorem{Assumption}{Assumption}
\newshadetheorem{Corollary}{Corollary}
\newshadetheorem{Definition}{Definition}
\newshadetheorem{Example}{Example}
\newshadetheorem{Remark}{Remark}
\newshadetheorem{Note}{Note}
\newshadetheorem{Hypothesis}{Hypothesis}
\newshadetheorem{Explanation}{Explanation}
\newtheorem{theorem}{Theorem}[section]
\newtheorem{lemma}[theorem]{Lemma}
\newtheorem{proposition}[theorem]{Proposition}

\newtheorem{corollary}[theorem]{Corollary}
\newtheorem{definition}[theorem]{Definition}

\newtheorem{remark}[theorem]{Remark}

\newcommand{\Tr}{\mathop{\mathrm{Tr}}}

\newcommand{\matr}{M}
\newcommand{\K}{\mathcal{K}}

\def\T{\mathbb{T}}
\def\mL{\mathrm{L}}
\def\Lb{\mathbf{L}}

\def\X{\mathrm{X}}

\def\F{\mathcal{F}}

\def\C{\mathbf{C}}

\def\0{\boldsymbol{0}}

\def\bu{\boldsymbol{u}}

\def\vb{\boldsymbol{v}}
\def\wb{\boldsymbol{w}}

\def\N{\mathbb{N}}

\def\M{\mathcal{M}}

\def\U{\mathcal{U}}
\def\bz{\boldsymbol{z}}

\def\Pr{\mathbb{P}}

\newcommand{\sphere}{\mathbb{S}}

\def\Eta{\boldsymbol{\eta}}
\def\bvarphi{\boldsymbol{\varphi}}
\def\bO{\mathcal{O}}

\def\bb{\boldsymbol{b}}
\def\Cb{\mathrm{C}}
\def\Nn{\mathcal{N}}
\def\matr{\mathbb{M}}
\def\ab{\boldsymbol{a}}

\newcommand{\abs}[1]{\left\vert#1\right\vert}

\newcommand{\R}{\mathbb{R}}

\newcommand{\loc}{\mathop{\mathrm{loc}}}
\newcommand{\di}{\mathrm{div}}

\renewcommand{\d}{\/\mathrm{d}\/}

\newcommand{\omegaa}{\boldsymbol{\omega}}
\newcommand{\rb}{\mathrm{b}}

\let\originalleft\left
\let\originalright\right
\renewcommand{\left}{\mathopen{}\mathclose\bgroup\originalleft}
\renewcommand{\right}{\aftergroup\egroup\originalright}

\newcommand{\ip}[2]{\fourIdx{}{0}{}{\!x}{\mathcal{ X}}}
\newcommand\coma[1]{{\color{red}#1}}
\newcommand\adda[1]{{\color{blue}#1}}

\newcommand\dela[1]{}

\hypersetup{colorlinks=true,%
	citecolor=blue,%
	filecolor=blue,%
	linkcolor=purple,%
}
\usepackage{graphicx}

\usepackage{orcidlink}

\usepackage{ragged2e}
\justifying

\usepackage{todonotes}

\usepackage{xpatch}
\makeatletter   
\xpatchcmd{\@tocline}
{\hfil\hbox to\@pnumwidth{\@tocpagenum{#7}}\par}
{\ifnum#1<0\hfill\else\dotfill\fi\hbox to\@pnumwidth{\@tocpagenum{#7}}\par}
{}{}
\makeatother    

\makeatletter
\def\l@subsection{\@tocline{2}{0pt}{4pc}{6pc}{}}
\def\l@subsubsection{\@tocline{3}{0pt}{8pc}{8pc}{}}
\makeatother

\makeatletter
\def\l@section{\@tocline{1}{12pt}{0pt}{}{\bfseries}}
\makeatother

\usepackage[utf8]{inputenc}
\usepackage[T1]{fontenc}
\mathtoolsset{showonlyrefs}
\usepackage{nomencl}
\makenomenclature
\usepackage{amssymb}

\usepackage{enumitem}

\title[Non-uniqueness for the stochastic Euler equations with a passive tracer]{Non-uniqueness for the stochastic incompressible Euler equations with a passive tracer}

\begin{document}
\vspace*{-4mm}
\maketitle
\begin{center}
	\author{Ashish Bawalia\footnote[1]{Department of Mathematics, Indian Institute of Technology Roorkee-IIT Roorkee, Haridwar Highway, Roorkee, Uttarakhand 247667, India.}\orcidlink{0009-0002-9141-2766}, Zdzis\l{}aw Brze\'zniak\footnotemark[2]$^\ast$\orcidlink{0000-0001-8731-6523} and Manil T. Mohan\footnotemark[1]\orcidlink{0000-0003-3197-1136}}
	\footnotetext[2]{Department of Mathematics, University of York, Heslington, YO10 5DD, York, United Kingdom. \\
		\textit{e-mail:} Ashish Bawalia: \email{ashish1@ma.iitr.ac.in; ashish1441chd@gmail.com}.\\
		\textit{e-mail:} Zdzis\l{}aw Brze\'zniak: \email{zdzislaw.brzezniak@york.ac.uk}.\\
		\textit{e-mail:} Manil T. Mohan: \email{maniltmohan@ma.iitr.ac.in; maniltmohan@gmail.com}.\\
		$\hspace{2mm} ^\ast$Corresponding author.
	}
\end{center}
\begin{abstract}
	In this work we investigate the phenomenon of pathwise non-uniqueness for the stochastic incompressible Euler equations with a passive tracer on the whole Euclidean space. The stochastic perturbations are interpreted as a transport noise and a linear multiplicative noise in the Stratonovich sense. In both cases, via classical transformations, we convert the SPDEs into PDEs with random coefficients. Using the Baire category method developed by De Lellis and Sz\'ekelyhidi Jr., we then construct infinitely many global-in-time weak solutions to the random PDEs in any spatial dimension greater than or equal to two. By applying the inverse transformations, we obtain pathwise non-uniqueness for the original SPDEs. Finally, we present an application of our result to the three-dimensional stochastic ideal MHD equations. This study can be regarded as a stochastic counterpart of Bronzi et al.~[\textit{Commun.~Math.~Sci.},~2015]. In particular, our non-uniqueness result in the random setting extends theirs from a constant energy profile equal to one to arbitrary positive, bounded, and continuous time-dependent profiles, originally established in two dimensions via the convex integration technique.
\end{abstract}
\vspace*{-2mm}
\tableofcontents


\section{Introduction}

The famous partial differential equations (PDEs) that describe the flow of an incompressible fluid were introduced by Euler in \cite{LE-57} and are known as the Euler equations, given by:
\begin{equation*}
	\left\{
	\begin{aligned}
		& \frac{\partial \bu(t)}{\partial t} + (\bu(t)\cdot \nabla) \bu(t) + \nabla \pi(t) = \0, \\
		& \mathrm{div \, } \bu(t) = 0.
	\end{aligned}
	\right.
\end{equation*}
For the last two centuries, the problem of well-posedness of this PDEs has been a challenging task and has attracted significant interest among mathematicians. The method of convex integration, introduced by Nash and Kuiper \cite{JN-54, NHK-55} for isometric embedding problems, later motivated De Lellis and Sz\'ekelyhidi \cite{CDL+LSJ-09} to apply it to the Euler equations. This eventually led to a novel line of research on the ill-posedness of fluid dynamics models, particularly the non-uniqueness of solutions, see Subsection \ref{Subsec-Pre} for a detailed literature review. Continuing in this spirit, we now turn to the model of interest, namely, \textit{the stochastic incompressible Euler system with a passive tracer perturbed by a transport and a multiplicative Gaussian white noise}.

Let  $(\Omega, \F, \{\F_t\}_{t\geq0}, \Pr)$ be a filtered probability space satisfying the usual conditions. We consider the following stochastic PDEs (SPDE), on $[0,\infty)\times \R^n$ ($n\ge 2$):
\begin{subequations}
	\begin{align}
		& \mathrm{d}\bu(t) + \mathrm{div \, }(\bu(t)\otimes \bu(t))\,\mathrm{d} t + \nabla \pi(t) \,\mathrm{d} t = - \Nn(\bu(t))\circ \d B(t), \label{DDEELMSN1}\\
		& \mathrm{d}\rb(t) + \mathrm{div \, } ( \rb(t)  \bu(t) )\,\mathrm{d} t = 0,\label{DDEELMSN2}\\
		& \mathrm{div \, } \bu(t) = 0,\label{DDEELMSN3} \\
		& \bu(0) = \0,\ \rb(0) = 0,\label{DDEELMSN4}
	\end{align}
\end{subequations}
where $\bu(t) = \bu(t,x) \in \R^n$ is the velocity field, $\pi(t)= \pi(t,x) \in \R$ is the pressure,
\begin{equation}\label{eqn-noise}
		\Nn(\bu)\circ \d B(t) := \left\{
		\begin{aligned}
			\nabla \bu\circ \d B(t) & = \sum_{j=1}^{n}\bigg( \frac{\partial u_1}{\partial x_j}, \dots,  \frac{\partial u_n}{\partial x_j}\bigg) \circ dB_j(t),\ \  \text{ transport},\\
			\gamma \bu\circ \d B(t) &= (\gamma u_1, \dots, \gamma u_n)\circ \d B(t),\ \ \text{ linear multiplicative},
		\end{aligned}
		\right.
\end{equation}
 $\gamma>0$ is a constant, $\circ$ means that the stochastic integral is understood in the sense of Stratonovich,  
\begin{equation}\label{eqn-def-B}
		B \; \left\{
		\begin{aligned}
			&\text{is a $\R^n$-valued Brownian motion in  the case of transport noise},\\
			&\text{is a $\R$-valued Brownian motion in the case of linear multiplicative noise}.
		\end{aligned}
		\right.
\end{equation}
and $\rb(t) = \rb(t,x)\in \R$ is the tracer.

The main aim of this article is to establish the existence of infinitely many bounded weak solutions (in the analytic sense defined below) to the problem \eqref{DDEELMSN1}--\eqref{DDEELMSN4}. 

\begin{remark}
	We conjecture that, for zero initial datum, any fluid dynamics model possessing infinitely many weak solutions retains this property when coupled with a transport equation and subjected to multiplicative and transport noise (acting on the velocity field). In this article, we verify this conjecture for the incompressible Euler equations in dimensions $n\ge2$.
\end{remark}

To this end, we first provide two definitions of weak solutions for both cases in It\^o form (using the Stratonovich–It\^o conversion).

\begin{definition}[Transport noise]\label{Weak-soln-T}
	We say that a pair of processes $$(\rb, \bu)\in \mathrm{C}_{w}([0,\infty); \mL^2(\R^n; \R\times \R^n))\cap\mL^\infty_{\mathrm{loc}}( [0,\infty) \times \R^n ; \R\times \R^n),\ \ \Pr\text{-a.s.}$$ is an analytically weak solution of the problem \eqref{DDEELMSN1}--\eqref{DDEELMSN4} 
	if and only if 
	\begin{itemize}
		\item the processes
		\begin{align*}
			t \mapsto \int_{\R^n} \rb(t) \varphi \,\mathrm{d} x\ \text{ and } \ t \mapsto \int_{\R^n} \bu(t)\cdot \bvarphi \,\mathrm{d} x,
		\end{align*}
		are continuous $\{\F_{t}\}_{t\geq 0}$-adapted for every function $\varphi \in \Cb^{1}(\R^n; \R)$ and vector field $\bvarphi \in \Cb^{2}(\R^n ; \R^n)$; 
		\item for any $t \geq 0$, $\Pr$-a.s. the following conditions hold:
		\begin{align}
			\int_{\R^n} \rb(t) \varphi \,\mathrm{d} x 
			& = \int_{0}^{t}\int_{\R^n}(\rb(s) \bu(s))\cdot \nabla\varphi \,\mathrm{d} x\,\mathrm{d} s,\\
			\int_{\R^n} \bu(t)\cdot\bvarphi \,\mathrm{d} x 
			& = \int_{0}^{t}\int_{\R^n} ((\bu(s) \otimes \bu(s)) : \nabla \bvarphi + \pi(s) \,\mathrm{div \,} \bvarphi)\,\mathrm{d} x\,\mathrm{d} s\\
			& \quad +  \int_{0}^{t} \int_{\R^n} \bu(s) \cdot \Delta \bvarphi \,\mathrm{d} x \,\mathrm{d} s  + \int_{0}^{t} \int_{\R^n}  (\nabla \varphi)^\top u(s) \,\mathrm{d} x \cdot \,\mathrm{d} B(s) \dela{\sum_{j=1}^{n} \int_{0}^{t} \int_{\R^n} u_i(s) \frac{\partial \varphi_i}{\partial x_j}  \,\mathrm{d} x \,\mathrm{d} B_j(s)}, \label{eqn-weak-form-T}
		\end{align}
	\end{itemize}
where $B$ denotes the $n$-dimensional Brownian motion, see \eqref{eqn-def-B}. \dela{the $u_i$ and $\varphi_i$ are the $i^{th}$-component of the vectors $\bu$ and $\bvarphi$.}
\end{definition}

\begin{definition}[Linear multiplicative noise]\label{Weak-soln}
	We say that a pair of processes $$(\rb, \bu)\in \mathrm{C}_{w}([0,\infty); \mL^2(\R^n; \R\times \R^n))\cap\mL^\infty_{\mathrm{loc}}( [0,\infty) \times \R^n; \R\times \R^n),\ \ \Pr\text{-a.s.}$$ is an analytically weak solution of the problem \eqref{DDEELMSN1}--\eqref{DDEELMSN4} 
	if and only if 
	\begin{itemize}
		\item the processes
		\begin{align*}
			t \mapsto \int_{\R^n} \rb(t) \varphi \,\mathrm{d} x\ \text{ and } \ t \mapsto \int_{\R^n} \bu(t)\cdot \bvarphi \,\mathrm{d} x,
		\end{align*}
		are continuous $\{\F_{t}\}_{t\geq 0}$-adapted for every function $\varphi \in \Cb^{1}(\R^n; \R)$ and vector field $\bvarphi \in \Cb^{1}(\R^n ; \R^n)$; 
		\item for any $t \geq 0$, $\Pr$-a.s. the following conditions hold:
		\begin{align}
			\int_{\R^n} \rb(t) \varphi \,\mathrm{d} x 
			& = \int_{0}^{t}\int_{\R^n}(\rb(s) \bu(s))\cdot \nabla\varphi \,\mathrm{d} x\,\mathrm{d} s,\\
			\int_{\R^n} \bu(t)\cdot\bvarphi \,\mathrm{d} x 
			& = \int_{0}^{t}\int_{\R^n} ((\bu(s) \otimes \bu(s)) : \nabla \bvarphi + \pi(s) \,\mathrm{div \,} \bvarphi)\,\mathrm{d} x\,\mathrm{d} s\\
			& \quad  + \frac{\gamma^2}{2} \int_{0}^{t} \int_{\R^n} \bu(s) \cdot \bvarphi \,\mathrm{d} x \,\mathrm{d} s  - \gamma \int_{0}^{t} \int_{\R^n} \bu(s)\cdot \bvarphi \,\mathrm{d} x \,\mathrm{d} B(s). \label{eqn-weak-form}
		\end{align}
	\end{itemize}
\end{definition}

\begin{remark}
	Observe that the solution introduced above is weak in the PDE sense, where partial derivatives are interpreted in the weak (or distributional) sense. However, they are strong in the stochastic sense, as the stochastic integral is considered on the original probability space.
\end{remark}


\subsection{Main results} Now we state the main results of this manuscript, which will be proved in the remainder of the article.

\begin{theorem}[{Transport and linear multiplicative noise}]\label{Main-result-Stochastic-PDE}
	Let $0<T<\infty$ be fixed. Let $\bO:= (0,T)\times \bO_x \subset [0,\infty)\times \R^n$ be a bounded domain and $h\in C([0,T])$ be a positive function bounded away from $0$. Then, there exists infinitely many weak solutions $(\rb, \bu)$ of the problem \eqref{DDEELMSN1}--\eqref{DDEELMSN4}, in the sense of Definitions \ref{Weak-soln-T} and \ref{Weak-soln}, such that $\Pr$-a.s.
	\begin{itemize}
		\item[(i)] for a.e. $(t,x) \in \bO$
		$$\abs{\bu(t,x)} = \left\{
		\begin{aligned}
			h(t),\ &\  \text{ transport case},\\
			\wtilde{h}(t),\ &\ \text{ linear multiplicative case},
		\end{aligned}
		\right.\ \text{ and }\ \; \abs{\rb(t,x)} =1,$$
		\item[(ii)] for a.e. $(t,x) \in ([0,\infty)\times \R^n) \backslash \bO$
		\begin{equation*}
			\bu(t,x) =0,\ \rb(t,x) =0\ \text{ and }\ \pi(t,x) =0,
		\end{equation*}
	\end{itemize}
	where $\wtilde{h}(t) = \theta^\prime(t) h(t)$ with $\displaystyle \theta(t)=\int_0^te^{-\gamma B(s)}\,\mathrm{d}s$.
\end{theorem}

Choosing $\wtilde{h}(t) = 1$, i.e., $h(t) = e^{\gamma B(t)}$, in linear multiplicative case, in Theorem \ref{Main-result-Stochastic-PDE} yields the following result of our interest. Observe that the function $e^{\gamma B(\cdot)}$ is in fact uniformly-continuous on $[0,T]$, for each fixed $0<T<\infty$. Furthermore, for each fixed path $\omega \in \Omega$, there exists a constant $K = K(\omega)>0$  such that 
\begin{equation}
	e^{\gamma B(t)}> K, \ \text{ for all }\ t \in [0,T].
\end{equation}

\begin{corollary}[{Linear multiplicative noise}]\label{Cor-h}
	Let $0<T<\infty$ be fixed, $\bO := (0,T)\times \bO_x\subset [0,\infty)\times \R^n$ be a bounded domain. Then, there exists infinitely many weak solutions $(\rb, \bu)$ of the problem \eqref{DDEELMSN1}--\eqref{DDEELMSN4}, in the sense of Definition \ref{Weak-soln}, such that $\Pr$-a.s.
	\begin{itemize}
		\item[(i)] for a.e. $(t,x) \in \bO$
		$$\abs{\bu(t,x)} = 1\ \text{ and }\ \abs{\rb(t,x)} =1,$$
		\item[(ii)] for a.e. $(t,x) \in ([0,\infty)\times \R^n) \backslash \bO$ $$\bu(t,x) = \0,\ \rb(t,x) =0\ \text{ and }\ \pi(t,x) =0.$$
	\end{itemize}
\end{corollary}

	Next, we state our main result for the PDEs with random coefficients, see Section \ref{Sec-Transforming to RPDE} for the transformation procedure used in both cases.
	\begin{theorem}
		\label{Main-result-RPDE}
		Let $0<T<\infty$ be fixed, $\bO:= (0,T)\times \bO_x \subset [0,\infty)\times \R^n$ be a bounded domain and $h\in C([0,T])$ be a positive function bounded away from $0$. Then, there exist infinitely many weak solutions 
		$$(\rb, \vb) \in \mL^\infty([0,\infty)\times \R^n; \R\times \R^n),$$ 
		in the sense of Definition \ref{Def-weak-rand}, to the problem \eqref{RDDEE} such that
		\begin{itemize}
			\item[(i)] $\abs{\vb(\theta(t),x)} = h(\theta(t))$ and $\abs{\rb(\theta(t),x)} =1$, for a.e. $(t,x) \in \bO$;
			\item[(ii)] $\vb(\theta(t),x) = \0$, $\rb(\theta(t),x) = 0$ and $p(\theta(t),x) =0,$ for a.e. $(t,x) \in [0,\infty)\times \R^n \backslash \bO$,
		\end{itemize}
	where $\displaystyle \theta(t)=\int_0^te^{-\gamma B(s)}\,\mathrm{d}s$.
	\end{theorem}
Note that, we do not require any time-transformation in the case of transport perturbation, therefore, the only change in Theorem \ref{Main-result-RPDE} is that $\theta(t)$ is replaced by $t$. Hence, we omit a separate statement for this case.
	\begin{remark}
		We emphasize that, for the random PDE corresponding to the SPDE with transport noise (see \eqref{RDDEE-T}), the only change in the above result is that $\theta(t)$ is replaced by $t$, since no time transformation is used.
	\end{remark}

	We conclude this part of the section by stating a result that directly connects the solutions of the SPDEs \eqref{DDEELMSN1}–\eqref{DDEELMSN4} to those of the random PDEs \eqref{RDDEE-T} and \eqref{RDDEE}.
	\begin{theorem}[Pathwise equivalence]
		Let $\bu$ be a weak solution to the system of equations \eqref{DDEELMSN1}-\eqref{DDEELMSN4}. Let $\vb$ be a process obtained via the transformation given in Section \ref{Sec-Transforming to RPDE}. Then, for $\Pr$-a.e. $\omega \in \Omega$, 
		\begin{equation*}
			\bu\ \text{ solves \ \eqref{DDEELMSN1}-\eqref{DDEELMSN4} }  \ \text{ if and only if }\ \; \vb\ \text{ solves }\ \eqref{RDDEE-T}\ \text{ and }\ \eqref{RDDEE},
		\end{equation*}
		where the inverse correspondence is given by the transformations provided in Section \ref{Sec-Rescale}.
	\end{theorem}
	
	Formally, the proof of the above result follows immediately from the transformations discussed in Sections \ref{Sec-Transforming to RPDE} and \ref{Sec-Rescale}.

The primary sources of inspiration for this work are the research article by Chiodaroli et al.~\cite{EC+EF+FF-21}, in which the authors proved the existence of infinitely many global-in-time weak solutions to the full Euler system driven by multiplicative white noise, the work of the second named author with Maurelli \cite{ZB+MM-20}, where \eqref{DDEELMSN1} is studied in vorticity form with more general transport noise, and the work of Bronzi et al.~\cite{CAB+CMFL+JHLN-15}, where non-uniqueness for the two-dimensional (2D) incompressible ideal flow with a passive tracer is established using the convex integration technique developed by De Lellis and Sz'ekelyhidi~\cite{CDL+LSJ-09}.


\subsection{Differences} In this manuscript, we develop the following novel differences:
\begin{itemize}
	\item[1.] From what we know, there are no existing non-uniqueness results for the incompressible Euler equations with a passive tracer under transport and linear multiplicative Stratonovich-type forcing on the whole space.
	
	\item[2.] In random settings, we generalize the results of Bronzi et al.~\cite{CAB+CMFL+JHLN-15} from constant energy profile in 2D to arbitrary positive continuous time-dependent energy profile in dimensions greater than or equal two.
	
	\item[3.] To construct infinitely many weak solutions to the random PDEs \eqref{RDDEE-T} and \eqref{RDDEE-T}, we employ the \textit{Baire category method} rather than the convex integration technique, which was left open in \cite{CAB+CMFL+JHLN-15}.
	
	\item[4.] As an application of this result (see Section \ref{ASMHD}), motivated by the deterministic counterpart \cite[Section 4]{CAB+CMFL+JHLN-15}, we show that there exist infinitely many weak solutions to the three-dimensional (3D) stochastic ideal magnetohydrodynamics (MHD) equations with both transport and linear multiplicative stochastic forcing.
	
	\item[5.] In the case of linear multiplicative noise, the result of Chiodaroli et al.~\cite{EC+EF+FF-21} applies to piecewise constant initial data on bounded domains in 2D and 3D only, whereas our result applies only to zero initial datum but on the whole Euclidean space in arbitrary dimensions $n\ge2$.
	
	\item[6.] In the transport noise case, Theorem \ref{Main-result-Stochastic-PDE} complements the result of Hofmanov\'a et al.~\cite[Theorem 1.3]{MH+TL+UP-24} (i.e., when the noise coefficient $\sigma_k = e_k$). Their work treats a more general class of transport noise that requires rough path theory, which we avoid here for simplicity.
\end{itemize}	

\begin{remark}
	We call $h$ an energy profile in the sense that, if $\vb$ is the velocity and $E$ is the corresponding kinetic energy, i.e. $E(t) = \|v(t)\|_{L^2}^2$, then $$E(t) = h^2(t)|\bO_x|, \ \text{ for a.e. }\ t \in [0,T],$$ where $\supp(v(t,\cdot)) \subseteq \bO_x \subset \R^n$, see Theorem \ref{Main-result-Stochastic-PDE}.
\end{remark}

\begin{remark}
	As mentioned by Chiodaroli et al.~in \cite[p.~1268]{EC+EF+FF-21}, due to the deterministic initial condition (in our case zero) and the stochastically rescaled time variable considered above (see \eqref{Time-transformation}), we do not require any \textit{oscillatory lemma} (cf.~\cite[Lemma 5.6]{DB+EF+MH-20}) to ensure the progressive measurability of the solution.
\end{remark}

\begin{remark}\label{Rmk-Comp}
		Let us observe that the work of Modena and Sz\'ekelyhidi Jr.~\cite{SM+LSJ-18} has already shown that the continuity equation admits infinitely many solutions on the torus in dimensions $n\ge3$. On the other hand, Hofmanov\'a et al.~\cite{MH+TL+UP-24} proved the existence of infinitely many solutions to the incompressible Euler equations perturbed by a general transport noise on 3D torus only. However, in this work with a simpler transport noise, we consider a coupled system consisting both the equations on the whole space $\R^n$ for any $n\ge 2$, see Theorem \ref{Main-result-Stochastic-PDE}.
\end{remark}

\begin{remark}
	Let us point out that this work considers the transport noise same as in one of the first groundbreaking results of Flandoli et al.~\cite{FF+MG+EP-10}, which is an another line of research known as \textit{regularization by noise} where the noise restores the well-posedness. However, our work focuses on the negation of this phenomenon in nonlinear problems. 
	 Nevertheless, our results are stronger in the sense that non-uniqueness holds pathwise rather than merely in law. Therefore, we wonder if it is possible to prove our results without using such transformations.
\end{remark}

\begin{remark}
	In \cite{HG+EN+CVV-14}, Holtz and Vicol shown the global-in-time existence of smooth pathwise solution to the incompressible Euler equations with a linear multiplicative It\^o noise on bounded smooth domain in 2D and 3D. On the other hand, we are dealing with  
	global-in-time weak solutions (in the analytic sense) of the incompressible Euler equations coupled with a passive tracer equation driven by a linear multiplicative noise of Stratonovich type on the Euclidean space, in any dimension greater than two. 
	Furthermore, they utilize the Leray-Helmholtz projection to 
	eliminate the pressure and solve the Euler equations for the velocity as an unknown, however, we deal with both the unknowns tracer, pressure and velocity simultaneously.
\end{remark}


\subsection{Previous works}\label{Subsec-Pre}
\subsubsection{Deterministic case}
In a seminal contribution, De Lellis and Sz\'ekelyhidi Jr.~\cite{CDL+LSJ-09} established the non-uniqueness of weak solutions to the incompressible Euler equations in all spatial dimensions $n \ge 2$, marking a major milestone in the study of ill-posedness for fluid dynamic models. Their work also provided clear and accessible reconstructions of several groundbreaking results previously obtained by Scheffer \cite{VS-93} in $\R^2$ and by Shnirelman \cite{AS-97} on the 2D torus.

Under periodic boundary conditions, and for arbitrary divergence-free initial data in $\mL^2$, Wiedemann \cite{EW-11} established the existence of infinitely many global-in-time weak solutions to the incompressible Euler equations with uniformly bounded energy. In a related direction, Sz\'ekelyhidi \cite{LSJ-11} constructed infinitely many admissible weak solutions corresponding to the classical vortex sheet initial data, assuming that the associated vorticity is not a bounded measure. Additionally, Choffrut and Sz\'ekelyhidi \cite{AC+LSJ-14} demonstrated the non-uniqueness of bounded weak stationary solutions to the incompressible Euler equations.

In \cite{CDL-08}, De Lellis constructed infinitely many entropy solutions to the full compressible Euler system on the whole space, demonstrating a fundamental instance of non-uniqueness. Subsequently, Chiodaroli \cite{EC-14} provided further evidence of ill-posedness by establishing the failure of uniqueness over finite time intervals for entropy solutions arising from smooth (continuously differentiable) initial density profiles in the spatially periodic setting.
In \cite{EC+CDL+OK-15}, Chiodaroli et al. showed that, for the 2D isentropic compressible Euler equations with pressure law $p(\rho) = \rho^2$ and classical Riemann initial data, there exist infinitely many admissible, bounded weak solutions. This result was extended by Chiodaroli and Kreml in \cite{EC+OK-14} to pressure laws of the form $p(\rho) = \rho^\gamma,$ with $\gamma \geq 1$. Moreover, in \cite{CDL+EC+OK-14}, De Lellis et al. demonstrated that, for appropriately chosen pressure functions and initial data, one can construct infinitely many bounded, admissible weak solutions that are non-self-similar and genuinely 2D in nature.

For the compressible barotropic Euler system, Feireisl \cite{EF-14} extended the result of Chiodaroli \cite{EC-14} and showed various counterexamples to well-posedness by combining the principle of maximal dissipation and the concept of admissible weak solutions.  In this spirit, several extended versions came, some of the important works are \cite{CDL+LSJ-13, CDL+LSJ-14}, where De Lellis and Sz\'ekelyhidi introduced a new definition of subsolutions with the help of Euler-Reynolds system and by understanding the turbulent flows as a superposition of Beltrami flows, which was introduced almost 30 years ago by Constantin and Majda \cite{PC+AM-88}, and they improved the regularity of solutions to H\"older continuous instead of bounded solutions.

These techniques were being adapted by several authors to various problems arising in mathematical physics, like Shvydkoy \cite{RS-11} constructed wild weak solutions to the general class of active scalar equations. Similar approach was adopted by \'Cordoba et al.~\cite{DC+DF+FG-11} in which they proved that the 2D incompressible porous media equation admits weak solutions with compact support in time and Sz\'ekelyhidi \cite{LSJ-12} constructed weak solutions to incompressible porous media equation with initial data given by the unstable interface. In \cite{PI+VV-15}, Isett and Vicol proved the non-uniqueness result for weak solutions, with H\"older regularity $\Cb^{1/9-}_{t,x}$, to inviscid active scalar equations with a divergence free drift velocity. In 2015, Bronzi et al.~\cite{CAB+CMFL+JHLN-15} extended the non-uniqueness result of \cite{CDL+LSJ-09} to incompressible ideal flow with passive tracer in 2D.

The Onsager's conjecture \cite{LO-49} was one of the famous open problem in fluid dynamics:
\begin{theorem}[{Onsager's conjecture}]
	Let $u$ be the weak solution of the Euler system.
	\begin{itemize}
		\item[(i)] If $u\in\Cb^{0, \alpha}$ with $\alpha>\frac{1}{3}$, then the kinetic energy is conserved.
		\item[(ii)] For $\alpha<\frac{1}{3}$, there are solutions $u\in\Cb^{0, \alpha}$ which do not conserve kinetic energy.
	\end{itemize}
\end{theorem}
The part (i) of the conjecture was proved by Constantin et al.~\cite{PC+WE+ET-91} and Eyink \cite{GLE-94} at the same time. Then, Cheskidov et al.~\cite{AC+PC+SF+RS-08} proved that  the kinetic  energy is conserved for velocities in the Besov space $B^{1/3}_{3, c(\N)}$, where $B^{1/3}_{3, c(\N)}$ defined as the class of all tempered distributions $v$ in $\R^3$ for which $\lim\limits_{q\to\infty} (2^q)^{1/3}\|\Delta_q v\|_{\mL^3} = 0$. Thereafter, Isett and Oh \cite{PI+SJO-16} provided a simple proof for weak solutions to the Euler equations on any compact Riemannian manifold, which conserve the kinetic   energy.

The first paper in the direction towards the proof of part (ii) was by De Lellis and Sz\'ekelyhidi \cite{CDL+LSJ-10, CDL+LSJ-13}, in which the authors produced  periodic weak solutions of the incompressible Euler equations on a 3D torus which dissipate the total kinetic energy and are H\"older continuous, for $\alpha<1/10$. Using the techniques of \cite{CDL+LSJ-13}, Daneri \cite{SD-14} proved the existence of infinitely many H\"older continuous initial vector fields starting from which there exists infinitely many H\"older continuous solutions with preassigned total kinetic energy, for $\alpha<1/16$.
In \cite{AC+CDL+LSJ-12} Choffrut et al.~extended the results of \cite{CDL+LSJ-13} to 2D setting. Then, Isett \cite{PI-12} took one step forward towards the proof by constructing global weak solutions to 3D incompressible Euler equations which are zero outside of a finite time interval and have velocity in the H\"older class $\Cb^{0,\alpha}$ for every $\alpha < 1/5$. In \cite{TB+CDL+PI+LSJ-15}, Buckmaster et al.~provided a shorter proof of Isett's previous result, adhering more to the original scheme and introduced some new devices. Furthermore, Buckmaster \cite{TB-15} showed how the later scheme can be adopted in order to prove the existence of non-trivial H\"older continuous solutions which for almost every time belong to the critical Onsager H\"older regularity $\Cb^{0,1/3 - \eps}$ and have compact temporal support.
Later in \cite{TB+CDL+LSJ-16}, Buckmaster el al.~improved the class of continuous periodic weak solutions that do not conserve the kinetic energy to $\mL^1_t \Cb_x^{1/3-\eps}$.

Finally in 2018, Isett \cite{PI-18} proved it by combining the method of convex integration relying on the Mikado flows, first introduced by Daneri and Sz\'ekelyhidi in \cite{SD+LSJ-17}, and a new ``gluing approximation'' technique. However, the solutions constructed in \cite{PI-18} are non-conservative, but Buckmaster et al.~\cite{TB+CDL+LSJ+VV-19} produced the solutions that dissipate the kinetic energy (strictly monotonic decreasing). We refer the reader to, for more recent literature on Euler equations, see \cite{DWB+SM+EST-24, SM+PQ-24, PS-24, FM-24, VG+HK+MN-24, VG+ROR-24, LCB+EC+RS-25, PB+EW-25}.

The convex integration techniques developed by De Lellis and Sz\'ekelyhidi can be considered as one of the important steps towards solving one of the million dollar problems related to the global solvability of 3D NSEs \cite{CF-06} proposed by the Clay Mathematics Institute in 2000. Firstly, Colombo et al.~\cite{MC+CDL+LDR-18} proved the ill-posedness of Leray solutions to the Cauchy problem for the hypodissipative NSEs, with exponent $\alpha<1/5$, by using the techniques from \cite{CDL+LSJ-13}.
Then, the intermittent convex integration techniques developed by Buckmaster and Vicol \cite{TB+VV-19-NSE} were able to establish that weak solutions of  3D NSEs are not unique in the class of weak solutions with finite kinetic energy. Moreover, the authors showed that the limit of a sequence of weak solutions with finite kinetic  energy of  3D NSEs is a H\"older continuous dissipative weak solution of  3D Euler equations.  For a detailed survey of these remarkable results, we refer the reader to \cite{TB+VV-19}.
In a ground breaking result \cite{AC+XL-22}, Cheskidov and Luo proved the non-uniqueness of weak solutions of NSEs in $\mL^p_t\mL^\infty_x$, with $p<2$ and any dimension $d \geq 2$, which is sharp in view of the classical Ladyzhenskaya-Prodi-Serrin condition $\frac{2}{p} + \frac{d}{q} \le 1$. {Indeed, if $d=3$ and exponents $p, q$ satisfies $\frac{2}{p} + \frac{3}{q} = 1$, then the weak-strong uniqueness holds in $\mL^p_t\mL^q_x$.}
In \cite{TB+MC+VV-22}, the authors have shown the non-uniqueness for a class of weak solutions to  NSEs which have bounded kinetic energy, integrable vorticity, and are smooth outside a fractal set of singular times with Hausdorff dimension strictly less than $1$.
In 2022, Colombo et al.~\cite{MC+LDR+MS-22} showed that distributional solutions, in $\mL^\infty_t\mL^2_x$, of incompressible NSEs, which are smooth in some open interval of times are meager and the Leray ones are a nowhere dense set.

Let us also mention a very recent result on the non-uniqueness of weak solutions to the ideal MHD equations in dimensions $n\ge3$ was established in \cite{MC+ZZ-25} using the approach of \cite{CDL+LSJ-09}, which also generalizes the application of the work \cite{CAB+CMFL+JHLN-15}.

\subsubsection{Stochastic case}
The first non-uniqueness result in the stochastic setting, mainly the pathwise non-uniqueness, motivated by the methods developed in \cite{CDL+LSJ-13}, produced by Breit et al.~\cite{DB+EF+MH-20}, where the authors showed the ill-posedness of the initial value problem in the class of weak solutions. They specifically proved that the Euler system admits infinitely many solutions up to a sequence of positive stopping times $\tau_m \to \infty$. Also, Chiodaroli et al.~\cite{EC+EF+FF-21} proved the existence of infinitely many pathwise global-in-time weak solutions to compressible Euler equations driven by multiplicative white noise on bounded domains in both 2D and 3D. Thereafter, in \cite{MH+RZ+XZ-22} Hofmanov\'a et al.~proved the existence and non-uniqueness of probabilistically strong and analytically weak solutions, of stochastic 3D incompressible Euler equations, defined up to a stopping time satisfying a version of  the energy inequality. Recently, in \cite{Lu+Lu+Zhu-25+}, L\"u et al.~proved a stochastic version of Onsager’s conjecture for the 3D Euler equations with additive noise on the periodic torus. For other important related works on stochastic Euler equations, see \cite{ZB+PS-01, HG+EN+CVV-14, ZB+FF+MM-16}.

On the other hand, in \cite{MH+RZ+XZ-24} by modifying the method of \cite{TB+VV-19-NSE}, Hofmanov\'a et al.~proved the law of analytically weak solutions of 3D Stochastic NSEs is not unique. In \cite{MH+RZ+XZ-24}, the authors focused on three type of stochastic perturbation driven by a Wiener process: an additive, a linear multiplicative and a nonlinear noise of cylindrical type. Furthermore, these solutions do not satisfy the corresponding energy inequality. This approach was adapted by Yamazaki \cite{KY-22} for 2D NSEs forced by a random noise with a diffusive term generalized via a fractional Laplacian that has an exponent less than $1$. In the sequel, Yamazaki \cite{KY-22-remarks} proved the non-uniqueness in law for the analytically weak solutions of 3D stochastic NSEs with the viscous diffusion in the form of a fractional Laplacian with exponent less than $1/5$. Furthermore, in \cite{KY-24}, the non uniqueness in law of the 3D NSEs forced by a random noise and diffused via a fractional Laplacian with a power in $(0,1/2)$.
In \cite{MH+RZ+XZ-23-trace-class}, Hofmanov\'a et al.~proved the existence of infinitely many global-in-time probabilistically strong solutions to the stochastic 3D incompressible NSEs driven by an additive stochastic forcing of trace class, for every divergence free initial condition in $\mL^2$ and non-uniqueness of the associated Markov processes. Thereafter, Hofmanov\'a et al.~\cite{MH+RZ+XZ-23} established the global-in-time existence and non-uniqueness of probablilistically strong solutions to 3D NSEs driven by space-time white noise by clubbing the rough path theory. In the case of 2D stochastic NSEs with derivative of space-time white noise, the existence of infinitely many stationary and ergodic solutions on torus showed by Lü and Zhu \cite{HL+XZ-24}. 
For further related works on the stochastic Navier–Stokes equations, see \cite{WC+ZD+XZ-24, HL+XZ-25, MH+RZ+XZ-25}; see also the survey by Yamazaki~\cite{KY-26}.

The rest of the manuscript is structured as follows.
The next section is devoted to the proof of Theorem \ref{Main-result-Stochastic-PDE} and is divided into four subsections. First, in Subsection \ref{Sec-Transforming to RPDE}, we transform the SPDE \eqref{DDEELMSN1}–\eqref{DDEELMSN2} into a random PDE using properties of Stratonovich calculus. The transport noise case is handled via translation by Brownian motion in the spatial variable, while the linear multiplicative noise case is treated using the Doss–Sussmann transformation together with a suitable time change. 
Next, in Subsection \ref{Sec-Geo setup}, we introduce the geometric setup for the PDE with random coefficients needed to incorporate the Baire category method, as well as convex integration. Subsection \ref{Sec-Baire category} defines the notion of a subsolution for the resulting random PDE and employs the Baire category method to construct infinitely many compactly supported weak solutions to \eqref{RDDEE}.
To complete the proof of the main result, in Subsection \ref{Sec-Rescale} we apply the inverse transformations and appropriate scaling to these solutions to obtain $\{\F_t\}_{t\geq 0}$-adapted solutions of the SPDE \eqref{DDEELMSN1}–\eqref{DDEELMSN2}. Finally, Appendix \ref{ASMHD} illustrates an application to the 3D stochastic ideal MHD equations, followed by Appendix \ref{CIT}, which presents the proof of Theorem \ref{Main-result-Random-PDE} via convex integration.


\section{Proof of Theorem \ref{Main-result-Stochastic-PDE}} 
This section is divided into four subsections, in which we develop the tools and ingredients needed to complete the proof of the main result of this article, namely Theorem \ref{Main-result-Stochastic-PDE}.

\subsection{Random PDE}\label{Sec-Transforming to RPDE}
In this subsection, we employ the flow transformation \cite[see p.~87]{UK+KY-25} (for general settings \cite[Theorem 3.3.2 on p.~93]{HK-90},) and the Doss-Sussmann transformation \cite{HD-77, JHS-78} to transform the SPDE \eqref{DDEELMSN1}-\eqref{DDEELMSN4} into a random PDE corresponding to the transport noise and linear multiplicative noise, respectively.
We then state the corresponding non-uniqueness result.

\subsubsection{Transport noise case}
We consider the transformation
\begin{equation*}
\vb(t,x) := \bu(t, x+B(t)),
\end{equation*}
where $B(t)$ is an $\R^n$-valued Brownian motion, see \eqref{eqn-def-B}.
Then, the usual chain-rule yields
\begin{equation*}
\d \vb(t, x) =  \d \bu(t, x+B(t)) + (\nabla \bu(t, x+ B(t))) \circ \d B(t).
\end{equation*} 
Thus, by substituting $\d \bu(t, x+B(t))$ from the equation \eqref{DDEELMSN1} gives
\begin{align*}
\d \vb(t, x) & =  - ((\bu\cdot \nabla)\bu)(t, x+B(t))\d t - \nabla \pi(t, x+B(t)) \d t\\
&\quad - (\nabla \bu(t, x+ B(t))) \circ \d B(t) + (\nabla \bu(t, x+ B(t))) \circ \d B(t)\\
& =  - ((\bu\cdot \nabla)\bu)(t, x+B(t))\d t - \nabla \pi(t, x+B(t)) \d t\\
& =  - ((\vb\cdot \nabla)\vb)(t, x)\d t - \nabla p(t, x) \d t,
\end{align*}
where $p(t,x) = \pi(t,x+ B(t))$.
Thus, with a slight abuse of notation and for simplicity in the subsequent analysis, we still denote $\rb(t,x+B(t))$ by $\rb(t,x)$ and deduce the tracer equations as follows
\begin{equation*}
\frac{\partial \rb(t, x)}{\partial t} + \nabla \rb(t, x) \cdot \vb(t, x) = 0.
\end{equation*}		
The transformed PDE with random coefficients is
\begin{equation}\label{RDDEE-T}
\left\{\begin{aligned}
& \vb_t(t) + \mathrm{div \, }(\vb(t)\otimes \vb(t) ) + \nabla p(t) = \0,\\
& \rb_t(t) + \mathrm{div \, }(\rb(t) \vb(t) = 0,\\
& \mathrm{div \, } \vb(t) = 0,\\
& \vb(0) = \0,\ \rb(0) = 0.
\end{aligned}\right.
\end{equation}
\begin{remark}\label{Rmk-abuse-1}
	The above abuse of notation allows us to treat both the transport and linear multiplicative random PDEs within a unified framework. The solution to the SPDEs is then recovered via the appropriate inverse transformations; see Section \ref{Sec-Rescale}.
\end{remark}

\subsubsection{Linear multiplicative noise case}
Similar to the above case, by using the transformation, motivated from \cite[Section 1.4]{ZB+UM+DM-19},
\begin{equation}\label{eqn-DS-trans}
\vb(t) := e^{\gamma  B(t)} \bu(t),\ \text{ for } \ t\ge 0,
\end{equation}
the SPDE \eqref{DDEELMSN1} can be transformed in the following random PDE:
\begin{align}\label{DDEE-1}
\frac{\partial  \vb(t)}{\partial t} = - e^{-\gamma B(t)} \mathrm{div \, } (\vb(t) \otimes \vb(t)) -  e^{\gamma B(t)} \nabla \pi(t).
\end{align}
Let us introduce, in the above expression, the time transformation, see \cite[p.~1272]{EC+EF+FF-21}, as follows:
\begin{align}\label{Time-transformation}
t \mapsto  \theta(t):= \int_{0}^{t}  e^{-\gamma  B(s)}\,\mathrm{d} s.
\end{align}
Then, the equation \eqref{DDEE-1} becomes
\begin{equation*}
e^{-\gamma  B(t)} \left[\frac{\partial  \vb(\theta)}{\partial \theta} +
\mathrm{div \, } (\vb(\theta) \otimes \vb(\theta)) + e^{2\gamma B(t)} \nabla \pi(\theta) \right] = \0.
\end{equation*}
Since $e^{-\gamma  B(t)} \ne 0$, it is immediate that
\begin{align}\label{REE}
\frac{\partial  \vb(\theta)}{\partial \theta} +
\mathrm{div \, } (\vb(\theta) \otimes \vb(\theta)) + \nabla p(\theta) = \0,
\end{align}
where $\displaystyle p(\theta) = e^{2\gamma B(t)} \pi(\theta) = \left( \theta^\prime(t)\right)^{-2}\pi(\theta)$. Similarly, we obtain the continuity equation \eqref{DDEELMSN2} as
\begin{equation*}
e^{-\gamma  B(t)}\frac{\partial \rb(\theta)}{\partial \theta} + \nabla \rb(\theta) \cdot \bu(\theta) = 0,
\end{equation*}
that implies
\begin{equation*}
\frac{\partial \rb(\theta)}{\partial \theta} + \nabla \rb(\theta) \cdot \vb(\theta) = \0.
\end{equation*}

For simplicity, we use $\rb_\theta$ and $\vb_\theta$ to denote the partial derivatives of $\rb$ and $\bu$ with respect to time $\theta$. By clubbing the above equation with \eqref{REE} and \eqref{DDEELMSN3}-\eqref{DDEELMSN4}, we arrive at the following  random incompressible Euler system with a passive tracer:
\begin{equation}\label{RDDEE}
\left\{\begin{aligned}
& \vb_\theta(\theta) + \mathrm{div \, }(\vb(\theta)\otimes \vb(\theta) ) + \nabla p(\theta) = \0,\\
& \rb_\theta(\theta) + \mathrm{div \, }(\rb(\theta) \vb(\theta) = 0,\\
& \mathrm{div \, } \vb(\theta) = 0,\\
& \vb(0) = \0,\ \rb(0) = 0,
\end{aligned}\right.
\end{equation}
where $\theta = \theta(t)$ given in \eqref{Time-transformation}.

Observe that, for a.e. fixed path $\omega \in \Omega$, the above-mentioned random PDEs \eqref{RDDEE-T} and \eqref{RDDEE} behave like deterministic PDEs. Therefore, we are now in a position to define a weak solution for these systems simultaneously.

\begin{definition}\label{Def-weak-rand}
A pair $(\rb, \vb)\in\mL_{ \loc}^{2}([0,\infty)\times \R^n; \R \times\R^n)$ is called a \emph{weak solution} of the incompressible Euler equations with a passive tracer \eqref{RDDEE}, if for a.e. $\theta\in [0, \infty)$ and for any test pair of functions $(\varphi,\bvarphi)\in\Cb_0^\infty((0,\infty)\times \R^n; \R \times \R^n)$ with $\di\, \bvarphi =0,$ the pair $(\rb, \vb)$ satisfies
\begin{equation}\label{weak-for}
\begin{aligned}
\int_{0}^\infty\int_{ \R^n} \rb\;\varphi_\theta\, \mathrm{d} x\,\mathrm{d}\theta + \int_{0}^\infty\int_{ \R^n} (\rb\vb)\cdot \nabla\varphi\, \mathrm{d} x\,\mathrm{d}\theta & = 0,\\
\int_{0}^\infty\int_{ \R^n} \vb\cdot\bvarphi_\theta\, \mathrm{d} x \,\mathrm{d}\theta + \int_{0}^\infty\int_{ \R^n} (\vb\otimes \vb): \nabla\bvarphi\, \mathrm{d} x\,\mathrm{d}\theta & = 0,\\
\int_{0}^\infty\int_{ \R^n} \vb\cdot \nabla\bvarphi\, \mathrm{d} x\,\mathrm{d}\theta & = 0.
\end{aligned}
\end{equation}
\end{definition}
\begin{remark}
Notice that the only change in the above definition, in the case of transport noise, is that the time variable $\theta$ is replaced by the usual time $t$, therefore, we omit a separate definition for this case.
\end{remark}
The following result yields infinitely many compactly supported weak solutions to the deterministic system \eqref{RDDEE} (for a.e. fixed path). As mentioned earlier, by using the method of Baire category, see Subsection \ref{Sec-Baire category}, we extend the result of Bronzi et al.~\cite[Theorem 1.1]{CAB+CMFL+JHLN-15} from a constant energy profile equal to one to arbitrary positive, bounded, and continuous time-dependent profiles, and to arbitrary dimensions $n\ge 2$.
\begin{theorem}\label{Main-result-Random-PDE}
Let $0<T<\infty$ be fixed, $\bO:= (0,T)\times \bO_x \subset [0,\infty)\times \R^n$ be a bounded domain and $h\in C([0,T])$ be a positive function bounded away from $0$. Then, there exist infinitely many weak solutions $$
(\rb, \vb) \in \mL^\infty([0,\infty)\times \R^n; \R\times \R^n),$$
in the sense of Definition \ref{Def-weak-rand}, to the problem \eqref{RDDEE} such that
\begin{itemize}
\item[(i)] $\abs{\vb(\theta(t),x)} = h(\theta(t))$ and $\abs{\rb(\theta(t),x)} =1$, for a.e. $(t,x) \in \bO$,
\item[(ii)] $\vb(\theta(t),x) = \0$, $\rb(\theta(t),x) = 0$ and $p(\theta(t),x) =0,$ for a.e. $(t,x) \in ([0,\infty)\times \R^n) \backslash \bO$,
\end{itemize}
where $\theta(\cdot)$ is as defined in \eqref{Time-transformation}.
\end{theorem}

\begin{remark}\label{Rmk-abuse-2}
For notational convenience, we use $t$ in place of $\theta$ up to Section \ref{Sec-Baire category}. This change of notation is not needed for the PDEs associated with the transport noise setting, see \eqref{RDDEE-T}, since no time transformation is involved in converting the SPDE into a random PDE. 
\end{remark}

\begin{remark}
	
Note that we are dealing with a random PDE, so the energy function $h(\cdot)$ may depend on paths of the Brownian motion $B$. For almost every fixed path in the underlying probability space, this dependence is deterministic, which is the setting considered in Corollary \ref{Cor-h}.
\end{remark}

\begin{remark}\label{Rmk-weak-c}
From classical results, see \cite[Theorem]{RT-01} or \cite[Lemma 8]{CDL+LSJ-10}, it follows that the distributional solution of \eqref{RDDEE-T} (or \eqref{RDDEE}) can be redefined on a set of measure zero in such a way that $\vb\in \mathrm{C}_{w}([0,\infty); \mL^2(\R^n; \R^n))$. 
\end{remark}

The proof of Theorem \ref{Main-result-Stochastic-PDE} follows directly from that of Theorem \ref{Main-result-Random-PDE}, combined with a reverse translation by Brownian motion in the transport noise case and a time inversion in the linear multiplicative case; see Subsection \ref{Sec-Rescale}. Consequently, the remainder of the paper, up to Subsection \ref{Sec-Rescale}, is devoted to establishing Theorem \ref{Main-result-Random-PDE}.


\subsection{Geometric setup}\label{Sec-Geo setup}
In this subsection, we collects geometric tools needed for the Baire category method as well as convex integration framework.
First note that, the incompressible Euler system with a passive tracer \eqref{RDDEE} can be naturally rewritten in Tartar's framework \cite{LT-79}, i.e.,  we can rewrite the system \eqref{RDDEE} as the following system of differential equations:
\begin{equation}\label{DE}
\left\{\begin{aligned}
& \vb_t + \mathrm{div \, }\bz + \nabla q = \0,\\
& \rb_t + \mathrm{div \, } \Eta = 0,\\
& \mathrm{div \, } \vb = 0,
\end{aligned}\right.
\end{equation}
with
\begin{align*}
q = p + \frac{\abs{\vb}^2}{n},\ \bz = \vb\otimes \vb - \frac{\abs{\vb}^2}{n}  I_n\ \text{ and } \Eta = \rb\vb,
\end{align*}
where $\bz(\cdot,\cdot) = \bz$ belongs to the space of symmetric trace-free $n\times n$ real matrices, denoted by 
\begin{equation}\label{eqn-S_0^n}
\matr_0^n := \left\{ A \in M_{n\times n}(\R) : A= A^\top\ \text{ and }\ \Tr A=0 \right\}.
\end{equation}
Motivated from \cite{CDL+LSJ-10}, let $h\in C([0,T])$ be a positive function that is bounded away from zero. Let us fix $t$ for a.e. $t\in [0,T]$. Then define the graph, or constraint set,  by:
$$\K_t := K_t \times [-1,1],$$ 
where
\begin{align}
K_t := \Big\{&(\wtilde{\rb}, \wtilde{\Eta}, \wtilde{\vb}, \wtilde{\bz}) \in \{-1,1\}\times \R^n\times\R^n\times\matr_0^n :\\
&\quad\wtilde{\bz}= \wtilde{\vb}\otimes \wtilde{\vb} - \frac{|\wtilde{\vb}|^2}{n}I_n,\ |\wtilde{\vb}| = h(t) \ \text{ and }\ \wtilde{\Eta} = \wtilde{\rb} \wtilde{\vb} \Big\}.\label{K}
\end{align}
\begin{remark}
Observe that, in definition of $K_t$, for a.e. fixed $t\in [0,T]$, we consider $h(t) >0$ and bounded away from $0$. So that the perturbation depends on the size of $K_t$ is non-vanishing.
\end{remark}
Next, we introduce an essential open set that will be used in the subsequent analysis.
\begin{equation}\label{eqn-U}
\U_t := \mathrm{int} (\K_t^\mathrm{co}) = \mathrm{int} (K_t^{\mathrm{co}} \times [-1,1]) \subset \mathbb{R} \times \mathbb{R}^{n} \times \mathbb{R}^{n} \times \mathbb{R}^{n \times n} \times \mathbb{R},
\end{equation}
where $K_t^{\mathrm{co}}$ represents the convex hull of $K_t$, and int stands for the topological interior of the set in $\R \times \R^n \times \R^n \times \matr_0^n \times \R$. Note that in \cite{CAB+CMFL+JHLN-15}, the authors considered the set $K=K_t$ with the $h(t)=1$, only for $n=2$.

A trivial, yet important observation, is the following result.

\begin{proposition}\label{prop-U}
For almost each fixed $t\ge0$, the set $\U_t$ is a bounded subset of $\mathbb{R} \times \mathbb{R}^{n} \times \mathbb{R}^{n} \times \mathbb{R}^{n \times n} \times \mathbb{R}$.
\end{proposition}
\begin{proof}
Let us fix a.e. $t\ge0$. Then, 
the proof is immediate from the fact the that set $K_t$ is bounded.
\end{proof}

\begin{lemma}\label{Lem-U}
For a.e. fixed $t\ge0$, $\0 \in \U_t$. Thus, $\U_t$ is non-empty.
\end{lemma}
The proof is deferred to the appendix, see Subsection \ref{Subsec-U}.

\begin{remark}\label{U-nonempty}
It is clear from \cite[Lemma 2.1]{CDL+LSJ-09} that any solution $(\rb, \Eta, \vb, \bz, q)$ of  the system \eqref{DE} whose image contained in $\K_t$ is a solution of the system \eqref{RDDEE}.
\end{remark}

Let us now introduce the following $(n+2)\times (n+1)$ matrix field:
\begin{align}\label{matrix-form}
U = \begin{pmatrix}
\bz + q I_n & \vb\\
\vb^\top & 0\\
\Eta^\top & \rb
\end{pmatrix}
\end{align}
and a new coordinate system $y = (t,x)\in
[0,\infty)\times \R^n$, where $x=(x_1, \cdots, x_n) \in\R^n$. In this setting, the equation \eqref{DE} reduces to
\begin{align}
\mathrm{div}_y\, U = \0.\label{Tartar-form}
\end{align}
Next, we set $\M_{(n+1)\times (n+1)}$ to be the set of \textit{symmetric} $(n+1)\times(n+1)$ matrices $Q$ such that $Q_{n+1,n+1} =0$ and suppose $\M_{(n+2)\times (n+1)}$ is the set of $(n+2)\times (n+1)$ matrices $A$ such that $(A_{i,j})_{1\leq i, j \leq n+1}\in \M_{(n+1)\times(n+1)}$. Observe that the following linear maps, motivated from \cite[Section 2]{CAB+CMFL+JHLN-15}, are isomorphisms:
\begin{align}
\R^n \times \matr_0^n \times \R \ni (\vb, \bz, q) &\mapsto \begin{pmatrix}
\bz + q I_n & \vb\\
v^\top & 0
\end{pmatrix} \in \M_{(n+1)\times(n+1)} ,\\
\R\times \R^n \ni (\rb, \Eta) &\mapsto \begin{pmatrix}
\Eta^\top & \rb
\end{pmatrix}\in  \R^{n+1}  ,\\
\R \times\R^n \times \R^n \times \matr_0^n \times \R \ni  (\rb, \Eta, \vb, \bz, q) &\mapsto \begin{pmatrix}
\bz + q I_n & \vb\\
\vb^\top & 0\\
\Eta^\top & \rb
\end{pmatrix} \in \M_{(n+2)\times(n+1)}.
\end{align}


\subsubsection{Plane wave solution} A plane wave solution of \eqref{Tartar-form} is a solution $U$, as in \eqref{matrix-form}, of the form
\begin{align}
U = U(y) = A h(y\cdot \xi), \label{plane-wave-soln}
\end{align}
where $h:\R \to \R$ and $A\in \M_{(n+2)\times(n+1)}$. Then, the \textit{wave cone is} the set of plane wave solution of \eqref{Tartar-form} for any $h$.
In our case, the wave cone is given by
\begin{align*}
\Lambda := \{A\in\M_{(n+2)\times(n+1)} : \mbox{there exists } \xi\in \R^{n+1} \backslash\{\0\} \text{ such that } A\xi = \0\},
\end{align*}
or, equivalently,
\begin{align}
\Lambda = \Bigg\{ &(\wtilde{\rb}, \wtilde{\Eta}, \wtilde{\vb}, \wtilde{\bz}, \wtilde{q})  \in \R\times \R^n\times\R^n\times\matr_0^n \times \R\\
& : \mbox{there exists } \xi \in \R^{n+1} \backslash \{\0\} \text{ such that }  \begin{pmatrix}
\wtilde{\bz} + \wtilde{q} I_n & \wtilde{\vb}\\
\wtilde{\vb}^\top & 0\\
\wtilde{\Eta}^\top & \wtilde{\rb}
\end{pmatrix}\xi = \0 \Bigg\}.\label{eqn-wave-cone}
\end{align}

As noted in Lemma \ref{Lem-U}, that $\0\in \U_t$, the next lemma demonstrate that the set $\U_t$ in fact contains an entire line segment. It can be viewed as a generalized result of \cite[Lemma 2.1]{CAB+CMFL+JHLN-15}, finding $\xi$ \eqref{eqn-xi} in the proof can be considered as one of the difficulty in this result, extending the result from the case $n=2$ to arbitrary dimension $n\geq2$.

\begin{lemma}\label{line-segment-lemma}
There exists a dimensional constant $C>0$ such that for each $(\rb, \Eta, \vb, \bz, q)\in\U_t$, there exists $(\overline{\rb}, \overline{\Eta}, \overline{\vb}, \overline{\bz}) \in \R\times \R^n\times\R^n\times\matr_0^n$ satisfying
\begin{itemize}
\item[(i)] $(\overline{\rb}, \overline{\Eta}, \overline{\vb}, \overline{\bz}, 0)\in \Lambda$;
\item[(ii)] the line segment with endpoints $(\rb, \Eta, \vb, \bz, q) \pm (\overline{\rb}, \overline{\Eta}, \overline{\vb}, \overline{\bz}, 0)\in \U_t$;
\item[(iii)]
$
|(\overline{\vb}, \overline{\rb})| \geq C ((h^2(t) + 1) - (\abs{\vb}^2 + \abs{\rb}^2))\ \text{ with }\ C = \frac{1}{4N\sqrt{2}}.$
\end{itemize}
\end{lemma}
\begin{proof}
\textbf{Step 1.} Let  us choose and take $z= (\rb, \Eta, \vb, \bz) \in \mathrm{int} \, K_t^{\mathrm{co}}$,  see \eqref{K} for $K_t$. Then, by Carath\'eodary's Theorem, see \cite[Theorem 17.1]{RTR-70},
the element $(\rb, \Eta, \vb, \bz)$ lies in the interior of a simplex in $\R\times \R^n\times\R^n\times\matr_0^n$ spanned by elements of $K_t$.
Therefore, we can write $z$ as
\[z := \sum_{i=1}^{N+1} \lambda_{i}z_i,\]
where for each $1\leq i\leq N+1$, $\lambda_i \in (0,1), \, z_i
\in K_t,\,  \sum\limits_{i=1}^{N+1} \lambda_i = 1$ and
\begin{equation}
N = 1+n+n+ \frac{n(n+1)}{2} -1 = \frac{n(n+5)}{2}\label{eqn-N}
\end{equation}
is the dimension of $\R\times \R^n\times\R^n\times\matr_0^n$. Assume that the coefficients are ordered so that $\displaystyle\lambda_1 = \max\limits_{1\leq i\leq N+1} \lambda_{i}$. Then, for any $j>1$, we have
\begin{align}\label{Lemma-1-1}
z \pm \frac{1}{2}\lambda_{j}(z_{j} - z_1) \in\mathrm{int}\, K_t^{\mathrm{co}}.
\end{align}
On the other hand, we write $$\displaystyle z - z_1 = \sum_{i=1}^{N+1} \lambda_{i}z_i - \sum_{i=1}^{N+1} \lambda_{i}z_1 = \sum_{i=2}^{N+1} \lambda_{i}(z_i - z_1),$$ 
in particular, by triangle inequality, we find
\begin{align}
\abs{(\vb - \vb_1, \rb - \rb_1)} & = \abs{\sum_{i=2}^{N+1} \lambda_{i}(\vb_j - \vb_1, \rb_j - \rb_1)}\\
& \leq \sum_{i=2}^{N+1} \sqrt{ \lambda_{i}^2(\abs{\vb_i - \vb_1}^2 + \abs{\rb_i - \rb_1}^2 )}\\
& \leq N \sqrt{\max_{2\leq i\leq N+1} \{\lambda_{i}^2(\abs{\vb_i - \vb_1}^2 + \abs{\rb_i - \rb_1}^2)\}}. \label{Lemma-1-2}
\end{align}
Now, let us choose and fix $j>1$ be such that 
\begin{align*}
\lambda_{j}^2(\abs{\vb_j - \vb_1}^2 + \abs{\rb_j - \rb_1}^2) := \max_{2\leq i\leq N+1} \{\lambda_{i}^2(\abs{\vb_i - \vb_1}^2 + \abs{\rb_i - \rb_1}^2)\}.
\end{align*}
Then, for each $i=2,\ldots,N+1$, the inequality \eqref{Lemma-1-2} becomes
\begin{align}\label{Lemma-1-3}
\abs{\vb_i - \vb_1}^2 + \abs{\rb_i - \rb_1}^2 \leq N^2\, \lambda_{j}^2(\abs{\vb_j - \vb_1}^2 + \abs{\rb_j - \rb_1}^2).
\end{align}

\noindent
\textbf{Step 2.} Next, let us suppose that
\begin{align*}
(\overline{\rb}, \overline{\Eta}, \overline{\vb}, \overline{\bz}) = \frac{1}{2}\lambda_j(z_j - z_1),
\end{align*}
in particular, we assert
\begin{align*}
\overline{\vb} = \frac{1}{2}\lambda_j(\vb_j - \vb_1)\ \text{ and }\ \overline{\rb}= \frac{1}{2}\lambda_j(\rb_j - \rb_1).
\end{align*}
Then, we calculate
\begin{align*}
(\rb, \Eta, \vb, \bz) \pm 2(\overline{\rb}, \overline{\Eta}, \overline{\vb}, \overline{\bz}) & = z \pm \lambda_{j}(z_{j} - z_1)\\
& = \left(\lambda_1 \mp \lambda_{j}\right) z_1+\lambda_2z_2 + \dots +\left( \lambda_j \pm  \lambda_j \right)z_j + \dots + \lambda_{N+1}z_{N+1}\\
& = \sum_{i=1}^{N+1}\mu_{i} z_i,
\end{align*}
where $$\mu_i = \begin{cases}
\lambda_i \mp \lambda_j,  \text{ for } i = 1,\\
\lambda_i \pm \lambda_j,  \text{ for } i = j,\\
\lambda_j,  \text{ for } i = 2,\dots,j-1, j+1, \dots, N+1.
\end{cases}$$
Since $\mu_i\in(0,1)$, for all $1\leq i\leq N+1$ and $\displaystyle \sum_{i=1}^{N+1}\mu_{i} = 1$, it follows that 
\begin{equation*}
(\rb, \Eta, \vb, \bz) \pm 2(\overline{\rb}, \overline{\Eta}, \overline{\vb}, \overline{\bz}) \in K_t,
\end{equation*}
which further implies
\begin{align*}
\left((\rb, \Eta,\vb,\bz) \pm 2(\overline{\rb}, \overline{\Eta}, \overline{\vb}, \overline{\bz})\right)^{\mathrm{co}} \subset K_t^{\mathrm{co}}.
\end{align*}
Thus, the line segment with endpoints $(\rb, \Eta, \vb, \bz) \pm (\overline{\rb}, \overline{\Eta}, \overline{\vb}, \overline{\bz})$
is contained in the $\mathrm{int}\, K_t^{\mathrm{co}}$. Hence, the line segment $(\rb, \Eta, \vb, \bz, q) \pm (\overline{\rb}, \overline{\Eta}, \overline{\vb}, \overline{\bz}, 0)$ is contained in $\U_t$, which proves (ii).
\vskip 1mm
\noindent
\textbf{Step 3.} It remains to show that $(\overline{\rb}, \overline{\Eta}, \overline{\vb}, \overline{\bz}, 0)\in\Lambda$, i.e., there exists a $\xi\in \R^{n+1} \backslash \{\0\}$ such that
\begin{align*}
\begin{pmatrix}
\overline{\bz} & \overline{\vb}\\
\overline{\vb}^\top & 0\\
\overline{\Eta}^\top & \overline{\rb}
\end{pmatrix} \xi = \frac{1}{2} \lambda_{j} \begin{pmatrix}
\vb_j \otimes \vb_j - \vb_1 \otimes \vb_1 & \vb_j - \vb_1\\
(\vb_j - \vb_1)^\top & 0\\
(\Eta_j - \Eta_1)^\top & \rb_j - \rb_1
\end{pmatrix}\xi = \0.
\end{align*}
Let us fix $\vb_j = (v_j^1, \ldots, v_j^n)$ and $\vb_1 = (v_1^1, \ldots, v_1^n)$, and then, define the non-zero vector as follows.
\begin{equation}\label{eqn-xi}
\xi := \begin{cases}
(-1, \underbrace{0, \ldots, 0,}_{(n-1)-\text{terms}} v_1^1), & \text{if } v_j^1 = v_1^1,\\
\Bigg(-\displaystyle{\frac{ \sum\limits_{i=2}^{n} (v_j^i - v_1^i) }{v_j^1 - v_1^1}, \underbrace{1, \ldots, 1,}_{(n-1)-\text{terms}} - \frac{ \sum\limits_{i=2}^{n} (v_j^1v_1^i - v_1^1 v_j^i) }{v_j^1 -v_1^1}\Bigg)}, & \text{if } v_j^1 \ne v_1^1,\\
\end{cases}
\end{equation}
so that
\begin{align*}
\begin{pmatrix}
(v_j^1)^2 - (v_1^1)^2 & \cdots & v_j^1 v_j^n - v_1^1 v_1^n & v_j^1 - v_1^1\\
\vdots & \ddots & \vdots & \vdots\\
v_j^n v_j^1 - v_1^n v_1^1 & \cdots & (v_j^n)^2 - (v_1^n)^2 & v_j^n - v_1^n\\
v_j^1 - v_1^1 & \cdots & v_j^n- v_1^n & 0\\
\rb_j v_j^1 - \rb_1 v_1^1& \cdots & \rb_j v_j^n - \rb_1 v_1^n & \rb_j - \rb_1
\end{pmatrix} \xi = \0.
\end{align*}
This implies $(\overline{\rb}, \overline{\Eta}, \overline{\vb}, \overline{\bz}, 0)\in\Lambda$. It completes the proof of part (i).

\noindent
\textbf{Step 4.} Finally, note from the definition of $K_t$, see \eqref{K}, that 
\begin{equation*}
\abs{\vb-\vb_1} \geq \abs{\vb_1} - \abs{\vb} = h(t) - \abs{\vb},\ \mbox{ for }\ |\vb_1|= h(t) \dela{\sphere^{n-1}},
\end{equation*}
and 
\begin{equation*}
\abs{\rb-\rb_1} \geq \abs{\rb_1} - \abs{\rb} = 1 - \abs{\rb},\ \text{ for }\ \rb_1\in\{-1,1\}.
\end{equation*}
By utilizing the above inequalities along with \eqref{Lemma-1-3}, we deduce
\begin{align}
\frac{1}{4\sqrt{2}N} ((h^2(t) +1) \dela{2} - (\abs{\vb}^2 + \abs{\rb}^2))
& = \frac{1}{4\sqrt{2}N} [(h(t)\dela{1} + \abs{\vb}) (h(t)\dela{1} - \abs{\vb}) + (1 + \abs{\rb}) (1 - \abs{\rb})]\\
& \leq \frac{1}{2\sqrt{2}N} [(h(t) - \abs{\vb}) + (1 - \abs{\rb})]\\
& \leq \frac{1}{2N} \sqrt{(h(t) - \abs{\vb})^2 + (1 - \abs{\rb})^2}\\
& \leq \frac{1}{2N} \sqrt{\abs{\vb - \vb_1}^2 + \abs{\rb - \rb_1}^2}\\
& \leq \frac{N}{2N} \lambda_j \sqrt{\abs{\vb_j - \vb_1}^2 + \abs{\rb_j - \rb_1}^2 } \\
& = |(\overline{\vb}, \overline{\rb})|.
\end{align}
Hence $\displaystyle |(\overline{\vb}, \overline{\rb})| \geq C ((h^2(t)+1) - (\abs{\vb}^2 + \abs{\rb}^2))$, with $\displaystyle C = \frac{1}{4N\sqrt{2}}$, and part (iii) follows. It conclude the proof.
\end{proof}
\begin{remark} Note that to prove part (i) of Lemma \ref{line-segment-lemma}, we cannot use the result \cite[Lemma 4.3]{CDL+LSJ-09} directly to show that $(\overline{\rb}, \overline{\Eta}, \overline{\vb}, \overline{\bz}, 0)\in \Lambda$, because that approach is applicable when the matrix in \eqref{matrix-form} is symmetric.
\end{remark}

The following result shows that the wave cone $\Lambda$, defined in \eqref{eqn-wave-cone}, contains sufficiently many directions. It can be regarded as a detailed generalization of \cite[Proposition 2.2]{CAB+CMFL+JHLN-15} in dimensions $n\geq2$.

\begin{proposition}\label{Perturbation-proposition}
Let $\begin{pmatrix}
\what{\bz} + \what{q} I_n & \what{\vb}\\
\what{\vb}^\top & 0\\
\what{\Eta}^\top & \what{\rb}
\end{pmatrix} =:\what{V} \in\Lambda$ be such that $\what{V}e_{n+1} \ne \0$, and consider the line segment $\sigma$ with endpoints $-\what{V}$ and $\what{V}$ in $\M_{(n+2)\times (n+1)}$. Then, there exists a constant $\alpha>0$ such that for any $\eps>0,$ there exists a smooth divergence-free matrix field
\begin{equation*}
V: [0,\infty)\times \R^n \to \M_{(n+2)\times (n+1)}\ \mbox{ given by }\	V(y) = \begin{pmatrix}
\bz(y) + q(y) I_n & \vb(y)\\
\vb^\top(y) & 0\\
\Eta^\top(y) & \rb(y)
\end{pmatrix}
\end{equation*}
where $\bz(\cdot)\in\matr_0^n$, $\vb(\cdot),\Eta(\cdot) \in \R^n$ and $\rb(\cdot), q(\cdot)\in \R$, with the properties
\begin{enumerate}[label=(P\arabic*), ref=P\arabic*]
\item\label{P1} $\supp (V) \subset B_1(\0)$,
\item\label{P2} $\mathrm{dist} (V(y), \sigma) < \eps\ \text{ for all }\ y\in B_1(\0)$,
\item\label{P3} $\displaystyle \int_{B_1(\0)} \abs{\vb(y)} \,\mathrm{d} y \geq \alpha\abs{\what{\vb}}\ \text{ and } \ \int_{B_1(\0)} \abs{\rb(y)} \,\mathrm{d} y \geq \alpha|\what{\rb}|$,
\end{enumerate}
where $\alpha>0$ is a dimensional constant 
and 
$B_1(\0)$ denotes the open ball of radius $1$ centered at $\0$ in $[0,\infty)\times \R^n$.
\end{proposition}

\begin{remark}
In the following proof,  we denote $B_r(y)$ as the  open ball of radius $r>0$ centered at $y$ in $[0,\infty)\times \R^n$.
\end{remark}
\begin{proof}[Proof of Proposition \ref{Perturbation-proposition}]
Let us first consider $\what{V} \in\Lambda$ of the form
\begin{align*}
\what{V} = \begin{pmatrix}
\what{U}\\
\what{W}^\top
\end{pmatrix}, \ \mbox{ with \ $\displaystyle \what{U} = \begin{pmatrix}
\what{\bz} + \what{q} I_n & \what{\vb}\\
\what{\vb}^\top & 0
\end{pmatrix}\in \matr_0^{n+1}$ \ and \ $\displaystyle \what{W} = \begin{pmatrix}
\what{\Eta}\\
\what{\rb}
\end{pmatrix}\in \M_{(n+1) \times 1}$. }
\end{align*}
Then, note that the $(n+1)\times(n+1)$ matrix $\what{U}$ is of the same type as in the incompressible Euler equations studied in \cite{CDL+LSJ-09}. Thus, by \cite[Proposition 3.2]{CDL+LSJ-09}, 
there exists a smooth divergence-free matrix field $U: \R^{n}\times[0,\infty) \to \M_{(n+1)\times(n+1)}$ which satisfies the required properties of Proposition \ref{Perturbation-proposition}.

Thus, it remains to show that there exists a smooth, divergence-free matrix $W\in \M_{(n+1) \times 1}$ satisfying \eqref{P1}, \eqref{P2} and \eqref{P3}. We begin by considering a special of $W$ and show that \eqref{P1} and \eqref{P2} hold. Using this particular case, we then establish properties \eqref{P1}-\eqref{P3} for general $W$.
\vskip 2mm
\noindent
\textbf{Step 1.} First, let us consider the case when
\begin{align}
\what{W}^\top e_{1} = 0\ \text{ and }\ \what{W}^\top e_{n+1} \ne 0,\ \mbox{ i.e., }\ \what{W} = \begin{pmatrix}
0 & \what{\eta_2} & \cdots & \what{\eta_n} & \what{\rb}
\end{pmatrix}^\top,\label{W-tilde}
\end{align}
where $\{e_i\}_{i=1}^{n+1}$ is the standard basis of $\R^{n+1}$.
Then, we fix a smooth cut-off function $\varphi : [0,\infty)\times \R^n \to \R$ which satisfies the following properties:
\begin{itemize}
\item[$\boldsymbol{\ast}$] $\abs{\varphi} \leq 1$,
\item[$\boldsymbol{\ast}$]  $\varphi =1 $ on $B_{\frac{1}{2}}(\0)$,
\item[$\boldsymbol{\ast}$]  $\supp (\varphi) \subset B_{1}(\0),$
\end{itemize}
and consider the mapping $W: [0,\infty)\times \R^n
\to \M_{(n+1) \times 1}$ given by
\begin{align}\label{W(y)}
W(y) := \frac{1}{N^2}\begin{pmatrix}
\sum\limits_{i=2}^{n}\partial_{1i}^2 (\what{\eta}_i \varphi(y) \sin (N y_1)) + \partial_{1(n+1)}^2 (\what{\rb} \varphi(y) \sin (N y_1))\\
-\partial_{11}^2 (\what{\eta}_2 \varphi(y) \sin (N y_1))\\
\vdots\\
-\partial_{11}^2 (\what{\eta}_n \varphi(y) \sin (N y_1))\\
-\partial_{11}^2 (\what{\rb} \varphi(y) \sin (N y_1))
\end{pmatrix},
\end{align}
where $y_1$ is the first coordinate of $y = (y_1, \ldots, y_{n+1})\in [0,\infty)\times \R^n$ and $N$ is same as in \eqref{eqn-N}.
The smoothness of $W$ follows from the fact that both $\varphi$ and $\sin (N y_1)$ are smooth. Furthermore, it easy to check that $W(y)$ is divergence-free. Since $\supp(\varphi) \subset B_{1}(\0)$, it implies that the map $W$ is supported in $B_{1}(\0)$, which immediately proves the required property \eqref{P1}.

Note from assumption, $\varphi(y) =1$ for $y\in B_{\frac{1}{2}}(\0)$, that 
\begin{align*}
W(y) = \what{W} \sin(Ny_1) \ \text{ for } \ y\in B_{\frac{1}{2}}(\0).
\end{align*}
In particular, we estimate
\begin{align}
\int_{B_{1}(\0)} \abs{W(y)^\top e_{n+1}} \,\mathrm{d} y & \geq \int_{B_{\frac{1}{2}}(\0)} \abs{W(y)^\top e_{n+1}} \,\mathrm{d} y\\
& \geq |\what{W}^\top e_{n+1}| \int_{B_{\frac{1}{2}}(\0)} \abs{\sin(Ny_1)} \,\mathrm{d} y\\
& \geq 2\alpha |\what{W}^\top e_{n+1}|\\
& = 2\alpha |\what{\rb}|,\label{p3 for W-tilde}
\end{align}
for some positive dimensional constant $\alpha = \alpha(n)$, for sufficiently large $N$.
Next, let us define
\begin{equation*}
\wtilde{W}(y): = \begin{pmatrix}
0 & \what{\eta}_2 \sin(Ny_1) & \cdots & \what{\eta}_n\sin(Ny_1) & \what{\rb}\sin(Ny_1)
\end{pmatrix}^\top=\widehat{W}\sin(Ny_1).
\end{equation*}
Then, calculate
\begin{equation}\label{2p8}
(W - \varphi\wtilde{W})(y) = \frac{1}{N^2}\begin{pmatrix}
\sum\limits_{i=2}^{n}\partial_{1i}^2 (\what{\eta}_i \varphi(y) \sin (N y_1)) + \partial_{1(n+1)}^2 (\what{\rb} \varphi(y) \sin (N y_1))\\
-\partial_{11}^2 (\what{\eta}_2 \varphi(y) \sin (N y_1)) - N^2\what{\eta}_2 \varphi(y) \sin(Ny_1)\\
\vdots\\
-\partial_{11}^2 (\what{\eta}_n \varphi(y) \sin (N y_1))- N^2\what{\eta}_n \varphi(y) \sin(Ny_1)\\
-\partial_{11}^2 (\what{\rb} \varphi(y) \sin (N y_1)) - N^2\what{\rb} \varphi(y) \sin(Ny_1)
\end{pmatrix}.
\end{equation}
Now, observe that
\begin{align*}
&	\sum_{i=2}^{n}\partial_{1i}^2 (\what{\eta}_i \varphi(y) \sin (N y_1))   + \partial_{1(n+1)}^2 (\what{\rb} \varphi(y) \sin (N y_1))\\
&  = \sum_{i=2}^{n} \what{\eta}_i [\partial_{1i}^2 \varphi(y) \sin(Ny_1) + N \partial_i \varphi(y) \cos(Ny_1)]\\
& \quad + \what{\rb} [\partial_{1(n+1)}^2 \varphi(y) \sin(Ny_1) + N \partial_{(n+1)} \varphi(y) \cos(Ny_1)].
\end{align*}
Furthermore, for $2 \leq j \leq n$, we have
\begin{align*}
-\partial_{11}^2 (\what{\eta}_j \varphi(y) \sin (N y_1)) - \what{\eta}_j N^2 \varphi(y) \sin(Ny_1)
= \what{\eta}_j[-\partial_{11}^2\varphi(y) \sin(Ny_1) - 2N \partial_1 \varphi(y) \cos(Ny_1)].
\end{align*}
Similarly, it follows that
\[-\partial_{11}^2 (\what{\rb} \varphi(y) \sin (N y_1)) - \what{\rb} N^2 \varphi(y) \sin(Ny_1)= \what{\rb} [-\partial_{11}^2\varphi(y) \sin(Ny_1) - 2N \partial_1 \varphi(y) \cos(Ny_1)].\]
Thus, by taking $\mL^\infty$-norm on both sides of \eqref{2p8} and using the triangle inequality, we obtain
\begin{align*}
\Vert W - \varphi\wtilde{W} \Vert_{\mL^\infty(B_{1}(\0))}
&\leq  \frac{C}{N}  \left\| \varphi\right\|_{\Cb^2(B_{1}(\0))}.
\end{align*}
By choosing $N$ sufficiently large, we deduce
\[	\Vert W - \varphi\wtilde{W} \Vert_{\mL^\infty(B_{1}(\0))}  \leq  \eps,\ \text{ for sufficiently small }\ \eps>0.\]
 On the other hand, since $\abs{\varphi} \leq 1$ and $\wtilde{W}(y)$ takes values in $\sigma_{\what{W}}$ (the line segment with endpoints $-\what{W}$ and $\what{W}$ in $\M_{(n+1)\times 1}$),
that means the image/range of $\varphi \wtilde{W}$ is also contained in $\sigma_{\what{W}}$. This shows that the image/range of $W$ is contained in the $\eps$-neighborhood of $\sigma_{\what{W}}$, which proves \eqref{P2} for $\what{W}$.
\vskip 2mm
\noindent
\textbf{Step 2.} Next, our aim is to prove a general case by reducing it to the settings of Step \textbf{1}.
\vskip 1mm
Let $\what{W}\in\M_{(n+1)\times1}$ be such that $\what{W}^\top e_{n+1} \ne 0$ and $\what{W}^\top f = 0$,
where $f\in\R^{n+1}\backslash\{\0\}$ is such that $\{f,e_{n+1}\}$ is linearly independent, otherwise the above assumptions will contradict each other.  Now, suppose $\{f_1, f_2, \ldots, f_{n+1}\}$ is a basis of $\R^{n+1}$ with $f_1 = f$ and $f_{n+1}  = e_{n+1}$, and consider the matrix $A\in \M_{(n+1)\times(n+1)}$ such that
\begin{align}\label{A-def}
Ae_i = f_i \ \text{ for } \ 1\leq i\leq n+1.
\end{align}
Then, the columns of $A$ forms a basis  of $\R^{n+1}$, which implies $\mathrm{det} (A) \ne 0$. But from \eqref{A-def}, we further infer  $Ae_{n+1}  = e_{n+1}$. Observe that the map $T:\R^{n+1} \to \R^{n+1}$ defined by
\[T X := (A^{-1})^{\top}  X\]
is a linear isomorphism on $\R^{n+1}$. Next, let us set	
\begin{align}\label{Setting-G-tilde}
\what{G} = A^\top  \what{W}\; (= T^{-1} (\what{W})).
\end{align}
Then, by the definition of $A$ and assumptions on $\what{W}$, we deduce
\begin{align*}
\what{G}_{1} = \what{G}^\top e_1  = (A^\top \what{W})^\top  e_{1} = \what{W}^\top  Ae_1 = \what{W}^\top f = 0,
\end{align*}
and
\begin{align*}
\what{G}_{n+1} = \what{G}^\top  e_{n+1} = (A^\top \what{W})^\top  e_{n+1} = \what{W}^\top  A e_{n+1} = \what{W}^\top  e_{n+1} \ne 0.
\end{align*}
Thus, by using Step 1, for a given $\eps>0$, we can construct a smooth map $G:
[0,\infty)\times \R^n \to \M_{(n+1)\times1}$ supported in $B_1(\0)$ such that
the image of $G$ is contained in $$B_{\eps\Vert T\Vert_{\mathrm{op}}^{-1}}(\tau_{\what{G}}) = B_{\eps\Vert (A^{-1})^{\top} \Vert_{\mathrm{op}}^{-1}}(\tau_{\what{G}}) = B_{\eps\Vert A^{-1}\Vert_{\mathrm{op}}^{-1}}(\tau_{\what{G}}),$$
where $\tau_{\what{G}}$ is the line segment with endpoints $-\what{G}$ and $\what{G}$ in $\M_{(n+1) \times 1}$, and $\Vert \cdot \Vert_{\mathrm{op}}$ denotes the operator norm.

Next, let $W$ be the $\M_{(n+1)\times1}$-valued map given by
\begin{align*}
W(y) = (A^{-1})^{\top}  G(A^\top y) = T(G(A^\top y)).
\end{align*}
Then, the isomorphism $T$ maps the line segment $\tau_{\what{G}}$ onto $\sigma_{\what{W}}$ (the line segment with endpoints $-\what{W}$ and $\what{W}$ in $\M_{(n+1)\times 1}$), i.e., $T(\tau_{\what{G}})  = \sigma_{\what{W}}$. Therefore, by the properties of the map $G,$ we have the following:
\begin{itemize}
\item[$\boldsymbol{\ast}$]  $W$ is supported in $(A^{-1})^{\top} (B_1(\0))$  and is smooth.
\item[$\boldsymbol{\ast}$]  Using the definition of $W$ and the fact that $G$ is divergence-free, we assert that $W$ is divergence-free, since
\begin{align*}
\int_{(A^{-1})^{\top} B_1(\0)} W(y) \cdot \nabla \psi(y) \,\mathrm{d} y
& = \int_{(A^{-1})^{\top} B_1(\0)} ((A^{-1})^{\top} G(A^\top y))\cdot\nabla \psi(y) \,\mathrm{d} y\\
& = (\mathrm{det}\, A)^{-1}\int_{B_1(\0)} G(z)\cdot\nabla \psi((A^{-1})^{\top} z) \,\mathrm{d} z\\
& = 0, \mbox{ for all  $\psi\in C_{0}^{\infty}(\R^{n+1})$.}
\end{align*}
\item[$\boldsymbol{\ast}$]  $W$ takes values in an $\eps$-neighborhood of the segment $\sigma_{\what{W}}$.
\end{itemize}
Furthermore, by utilizing the transformation $A^\top y = z$, the estimate \eqref{p3 for W-tilde} for $G(z)$ and the standard property $\mathrm{det} (A) = \mathrm{det} (A^\top )$, we calculate
\begin{align}
\int_{(A^{-1})^{\top} B_1(\0)} \abs{W^\top (y)e_{n+1}} \,\mathrm{d} y
& = 	\int_{(A^{-1})^{\top} B_1(\0)} \abs{((A^{-1})^{\top}  G(A^\top y))^\top  e_{n+1}} \,\mathrm{d} y \\
&  = \int_{B_1(\0)} \abs{((A^{-1})^{\top}  G(z))^\top  e_{n+1}} \frac{\,\mathrm{d} z}{\abs{\mathrm{det}(A^\top )}}\\
& \geq \frac{2\alpha|((A^{-1})^{\top}  \what{G})^\top  e_{n+1}|}{\abs{\mathrm{det}(A^\top )}}\\
&  =  \frac{ 2\alpha|((A^{-1})^{\top}  A^\top \what{W})^\top  e_{n+1}|}{\abs{\mathrm{det}(A)}}\\
&  =   \frac{2\alpha |\what{W}^\top  e_{n+1}|}{\abs{\mathrm{det}(A)}}. \label{estimate-W}
\end{align}
To complete the proof, we employ a covering argument, i.e., there exists a finite number of points $y_k\in B_1(\0)$ and radii $r_k>0$, for $k\in I,$ where $I$ is an index set, such that the rescaled and translated collection $\{(A^{-1})^{\top}  B_{r_k}(y_k)\}_{k\in I}$ is pairwise disjoint, contained in $B_1(\0)$, and
\begin{align}\label{convering of unit ball}
\sum_{k\in I} \abs{(A^{-1})^{\top}  B_{r_k}(y_k)} \geq \frac{1}{2}\abs{B_1(\0)},
\end{align}
where $|\cdot|$ denotes the $(n+1)$-dimensional Lebesgue measure. 
Now, let us consider 
\begin{equation}
W_k(y) :=  W\left(\frac{y-y_k}{r_k}\right)\ \text{ and } \ \accentset{\circ}{W} := \sum_{k\in I} W_{k}.
\end{equation}
Since $W_k$ is smooth and supported in $(A^{-1})^{\top}  B_{r_k}(y_k)$  for each $k\in I$, it implies that $\accentset{\circ}{W} : \R^{n}\times[0,\infty) \to \M_{(n+1)\times 1}$ is smooth, satisfy \eqref{P1} and \eqref{P2}. Furthermore, by utilizing the transformations $\displaystyle\frac{y-y_k}{r_k} = z$ and $A^\top z = q$, the estimate \eqref{estimate-W}, along with  the properties $\displaystyle\frac{| B_{r_k}(y_k)|}{|B_1(\0)|} = \frac{{\alpha(n+1)}r_{k}^{n+1}}{{\alpha(n+1)}}$ and $|\mathrm{det}{A}| = |\mathrm{det}{A^\top}|$, it follows that
\begin{align}
\int_{B_1(\0)} \big|\accentset{\circ}{W}^\top (y)e_{n+1}\big| \,\mathrm{d} y
& = \sum_{k\in I}\int_{(A^{-1})^{\top}  (B_{r_k}(y_k))} \abs{W^\top _k(y)e_{n+1}}\,\mathrm{d} y\\
& = \sum_{k\in I} \int_{(A^{-1})^{\top}  (B_{r_k}(y_k))} \abs{ W^\top \left(\frac{y-y_k}{r_k}\right) e_{n+1}} \,\mathrm{d} y\\
& = \sum_{k\in I}\int_{(A^{-1})^{\top} (B_{r_k}(y_k))} \bigg|\left((A^{-1})^{\top}  B\left(A^\top  \left(\frac{y-y_k}{r_k}\right)\right)\right)^\top  e_{n+1}\bigg| \,\mathrm{d} y\\
& = \sum_{k\in I} \int_{(A^{-1})^{\top} (B_1(\0))} \abs{((A^{-1})^{\top}  G(A^\top z))^\top  e_{n+1}} r_{k}^{n+1} \,\mathrm{d} z\\ 
& = \sum_{k\in I} \int_{B_1(\0)} \abs{((A^{-1})^{\top}  G(q))^\top  e_{n+1}} \frac{r_{k}^{n+1}}{\abs{\mathrm{det}(A^\top )}} \,\mathrm{d} q\\ 
& \geq \sum_{k\in I} \frac{2\alpha |\what{W}^\top  e_{n+1}|}{\abs{\mathrm{det}(A)}}r_{k}^{n+1}\\ 
& = \sum_{k\in I} \frac{2\alpha |\what{W}^\top  e_{n+1}| }{\abs{\mathrm{det}(A)}} \frac{\abs{ B_{r_k}(y_k)}}{\abs{B_1(\0)}}\\
& = \frac{2\alpha |\what{W}^\top  e_{n+1}| }{{\abs{\mathrm{det}(A)}}} \frac{\sum_{k\in I} {\abs{\mathrm{det}(A^\top )}}\abs{ (A^{-1})^{\top} (B_{r_k}(y_k))}}{\abs{B_1(\0)}}\\ 
& = \frac{2\alpha |\what{W}^\top  e_{n+1}| }{\abs{B_1(\0)}} \sum_k \abs{ (A^{-1})^{\top} (B_{r_k}(y_k))} .\label{2p11}
\end{align}
Hence by utilizing the estimate \eqref{convering of unit ball} in \eqref{2p11}, we infer
\begin{align*}
\int_{B_1(\0)} \abs{\rb(y)} dy = \int_{B_1(\0)} |\accentset{\circ}{W}^\top (y)e_{n+1}| dy  \geq  \alpha|\what{W}^\top  e_{n+1}| = \alpha|\what{\rb}|,
\end{align*}
i.e., \eqref{P3} is satisfied, which completes the proof.
\end{proof}


\subsection{Baire category method}\label{Sec-Baire category}

In this subsection, we define the concept of subsolutions and obtain certain estimates and convergence results. Then, we employ the Baire category method to generate infinitely many weak solutions to the PDE \eqref{RDDEE} for almost all fixed paths.

For simplicity, whenever $\bO\subseteq [0,\infty)\times \R^n$, we use the following notation throughout this subsection:
\begin{align*}
\Cb^{\infty}(\bO; \R \times \R^n \times \R^n \times \matr_0^n \times \R) & := \C^{\infty}(\bO),\\
\mL^2(\bO; \R \times \R^n \times \R^n \times \matr_0^n \times \R) & := \Lb^2(\bO),\\
\mL^\infty(\bO; \R \times \R^n \times \R^n \times \matr_0^n \times \R) & := \Lb^\infty(\bO).
\end{align*}

\subsubsection{Subsolutions}\label{Sec-Subsolutions}

Let us recall that $\U_t$ is a bounded subset of $\R \times \R^n \times \R^n \times M_0^n \times \R$, see Proposition \ref{prop-U}. First, we define a complete metric space with the help of the following subsolutions space:

Let $\X_0$ denote the set of functions $(\rb, \Eta, \vb, \bz, q) \in \C^{\infty}([0,\infty)\times\R^n)$ that satisfies
\begin{itemize}
\item[(i)] $\supp (\rb, \Eta, \vb, \bz, q) \subset \bO$,
\item[(ii)] $(\rb, \Eta, \vb, \bz, q)$ solves \eqref{DE} in $[0,\infty)\times \R^n$,
\item[(iii)] $(\rb,\Eta,\vb, \bz,q)(t,x)\in\U_t$ for all $(t,x)\in[0,\infty)\times \R^n$.
\end{itemize}
We equip $\X_0$ with the $\Lb^\infty$-$w^{\ast}$ topology and then define $\X,$ the closure of $\X_0$ in this topology.

The following result is, motivated from \cite[Lemma 4.4]{CDL+LSJ-09}, one of the important results of this work. It establishes that every element of $\X$ satisfies \eqref{eqn-mod-prop}, is a compactly supported weak solution to the problems \eqref{RDDEE-T} as well as \eqref{RDDEE}.

\begin{lemma}\label{Equivalent-lemma}
The set $\X$ with $\Lb^\infty$-$w^{\ast}$ topology is a non-empty compact metrizable space. Moreover, if $(\rb, \Eta, \vb, \bz, q)\in\X$ is such that
\begin{equation}\label{eqn-mod-prop}
\abs{\vb(t,x)} = h(t) \ \mbox{ and }\  \abs{\rb(t,x)} = 1 \ \mbox{ for a.e. } \ (t,x) \in \bO,
\end{equation}
then $\big(\vb,\rb,p\big),$ where $p:= q -\frac{\abs{\vb}^2}{2}\dela{ = q -\frac{h^2(t)}{2}},$ is a weak solution of \eqref{RDDEE-T} (or \eqref{RDDEE}) such that
\begin{equation*}
\vb(t,x) = 0,\ \rb(t,x) = 0\ \mbox{ and }\ p(t,x) = 0\ \mbox{ for all }\ (t,x) \in ([0,\infty)\times \R^n) \backslash\bO.
\end{equation*}
\end{lemma}
\begin{proof}
As a first step, let us select and fix a.e. $t\in [0,T]$ and show that $\X$ with $\Lb^\infty-w^{\ast}$ topology is a non-empty compact metrizable space. 

\vspace{1mm}
\noindent
\textbf{Step 1.} Observe from Lemma \ref{Lem-U} that $\0\in\U_t$, which implies $\X$ is non-empty. Furthermore, the boundedness of $\bO$ yields that $\X$ is a bounded and closed subset of $\Lb^\infty(\bO)$, i.e., $\X$ is $w^\ast$-compact. Therefore, by \cite[Theorem 3.16]{WR-73}, $\X$ with the $w^\ast$ topology becomes a compact metrizable space.

Next, we aim to show that each element $(\rb, \Eta, \vb, \bz, q)\in\X$ satisfying \eqref{eqn-mod-prop} is a weak solution of \eqref{RDDEE-T} (or \eqref{RDDEE}) in the sense of Definition \ref{Def-weak-rand} and has compact support.

\vspace{1mm}
\noindent
\textbf{Step 2.} Note from \eqref{eqn-U} that $\overline{\U_t}$ is a compact convex set. Thus, every $(\rb, \Eta, \vb, \bz, q)\in\X$ solves \eqref{DE}, takes values in $\overline{\U_t}$, and is supported in $\overline{\bO}$.
In particular, $(\rb, \Eta, \vb, \bz)(t,x) \in K_t^{\mathrm{co}}$ for almost every $(t,x)\in \bO$.
Therefore, from the definition of $K_t
$, see \eqref{K}, it implies that 
\begin{align*}
(\rb, \Eta, \vb, \bz)(t,x) \in K_t\
\mbox{ if and only if }\ \abs{\vb(t,x)} = h(t)\ \mbox{ and }\ \abs{\rb(t,x)} = 1.
\end{align*}
Hence, it follows from Remark \ref{U-nonempty} that $(\rb, \Eta, \vb, \bz, p)$ solves \eqref{RDDEE-T} (or \eqref{RDDEE}), which concludes the proof.
\end{proof}

The following result is a generalization of \cite[Lemma 2.2]{CAB+CMFL+JHLN-15} from the case of constant energy equal to one to a positive, bounded, continuous, time-dependent energy profile $h(\cdot)$ satisfying the assumptions of Theorem \ref{Main-result-Random-PDE}. In the sequel, $\abs{\bO_x}$ and $\abs{\bO}$ denote the $n$- and $(n+1)$-dimensional Lebesgue measures of $\bO_x$ and $\bO$, respectively.

\begin{remark}\label{Rmk-h-cont}
In the proof of the perturbation Lemma \ref{Convergence-lemma}, particularly in Step 3, we use the fact that $h(\cdot)$ is defined pointwise, which in turn requires continuity.
\end{remark}
\begin{lemma}\label{Convergence-lemma}
There exists a constant $\beta>0$ with the following property:

Given $(\rb_0, \Eta_0, \vb_0, \bz_0, q_0)\in\X_0,$ there exists a sequence $(\rb_k, \Eta_k, \vb_k, \bz_k, q_k)\in\X_0$ such that
\begin{align}
&\Vert \vb_k \Vert_{\mL^2(\bO)}^2 + \Vert \rb_k \Vert_{\mL^2(\bO)}^2\\
& \geq \Vert \vb_0 \Vert_{\mL^2(\bO)}^2 + \Vert \rb_0 \Vert_{\mL^2(\bO)}^2 + \beta \big((\|h\|_{L^2(0,T)}^2|\bO_x|+\abs{\bO}) - (\Vert \vb_0 \Vert_{\mL^2(\bO)}^2 + \Vert \rb_0 \Vert_{\mL^2(\bO)}^2)\big)^2,\label{eqn-inequality}
\end{align}
and
\begin{equation}\label{3p1}
(\rb_k, \Eta_k, \vb_k, \bz_k, q_k)  \xrightharpoonup[k\to\infty]{} (\rb_0, \Eta_0, \vb_0, \bz_0, q_0)\ \ \mbox{ w$^\ast$ in }\ \mL^\infty(\bO).
\end{equation}
\end{lemma}
\begin{proof}
\textbf{Step 1.} Let us choose and fix $z_0 := (\rb_0, \Eta_0, \vb_0, \bz_0, q_0) \in\X_0$. Then, $\mathrm{Im}(z_0) = \{(t,x)\in\bO : (\rb_0, \Eta_0, \vb_0, \bz_0, q_0)(t,x)\} \subset \U_t$ is a compact set. By applying Lemma \ref{line-segment-lemma} to each element of $\mathrm{Im}(z_0)$, we obtain that there exists a direction
\begin{align*}
\overline{z}(t,x) := (\overline{\rb}(t,x), \overline{\Eta}(t,x), \overline{\vb}(t,x), \overline{\bz}(t,x), 0)\in \Lambda
\end{align*}
such that the line segment with endpoints $z_0(t,x) \pm \overline{z}(t,x)$ is contained in $\U_t$, and
\begin{align}
|\overline{\vb}(t,x)| + |\overline{\rb}(t,x)|  \geq C ((h^2(t)+1) - (\abs{\vb_0(t,x)}^2 + \abs{\rb_0(t,x)}^2)).\label{Convergence-lemma-1}
\end{align}
Together with the uniform continuity of $(\rb_0, \Eta_0, \vb_0, \bz_0, q_0)(\cdot,\cdot)$, it follows that there exists $\eps>0$ such that for any $(t,x), (t_0,x_0)\in \bO$ with
\[\abs{x- x_0} + \abs{t-t_0} < \eps,\]
and the $\eps$-neighborhood of the line segment with endpoints $z_0(t,x) \pm \overline{z}(t_0,x_0)$ is also contained in $\U_t$.
\vskip 2mm
\noindent
\textbf{Step 2.} Fixing $(t_0,x_0) \in \bO$ and applying Proposition \ref{Perturbation-proposition} to
\begin{align*}
(\rb_0, \Eta_0, \vb_0, \bz_0, q_0)(t_0, x_0) \in \Lambda,
\end{align*}
it follows that, for every $\eps>0$ there exists a smooth solution $(\rb, \Eta, \vb, \bz, q)$ of \eqref{DE} satisfying properties \eqref{P1}--\eqref{P3}. Furthermore, for any $r<\eps,$ consider
\begin{align*}
(\rb_r, \Eta_r, \vb_r, \bz_r, q_r)(t,x) = (\rb, \Eta, \vb, \bz, q) \left(\frac{t-t_0}{r}, \frac{x-x_0}{r}\right).
\end{align*}
Then, $(\rb_r, \Eta_r, \vb_r, \bz_r, q_r)$ is also a solution of \eqref{DE}, which satisfy the following properties:
\begin{enumerate}[label=(p\arabic*), ref=p\arabic*]
\item\label{p1} $\supp (\rb_r, \Eta_r, \vb_r, \bz_r, q_r) \subset B_r(t_0,x_0) \subset [0,\infty)\times \R^n$,
\item\label{p2} $\mathrm{Im}(\rb_r, \Eta_r, \vb_r, \bz_r, q_r)$ is contained in the $\eps$-neighborhood of the line segment with endpoints $\pm (\overline{\rb}, \overline{\Eta}, \overline{\vb}, \overline{\bz}, 0)(t_0,x_0)$,
\item\label{p3} $\displaystyle\int_{B_r(t_0,x_0)} \abs{\vb_r(t,x)} \,\mathrm{d} x\,\mathrm{d} t \geq \alpha\abs{\overline{\vb}(t_0,x_0)}\abs{B_r(t_0,x_0)}$\\  and $\displaystyle \int_{B_r(t_0,x_0)} \abs{\rb_r(t,x)} \,\mathrm{d} x\,\mathrm{d} t \geq \alpha\abs{\overline{\rb}(t_0,x_0)}\abs{B_r(t_0,x_0)}$.
\end{enumerate}
From the above three properties, it is clear that the line segment with endpoints $z_0(t,x) \pm \overline{z}(t_0,x_0)$ is contained in $\U_t$, i.e.,  for any $ r<\eps,$ we have
\begin{align*}
(\rb_0, \Eta_0, \vb_0, \bz_0, q_0)+(\rb_r, \Eta_r, \vb_r, \bz_r, q_r) \in \X_0.
\end{align*}
\textbf{Step 3.}  Now, by utilizing the uniform continuity of $\vb_0$ and $\rb_0$, we can choose a radius $r_0>0$ such that, for any $r<r_0$, there exists a finite family of pairwise disjoint balls $B_{r_j}(t_j,x_j) \subset \bO$ with $r_j<r$ the following inequality holds:
\begin{align}
& \int_{\bO} ((h^2(t)+1) - (\abs{\vb_0(t,x)}^2 + \abs{\rb_0(t,x)}^2))\,\mathrm{d} x \,\mathrm{d} t\\
& \leq 2 \sum_{j} \big((h^2(t_j)+1) - (\abs{\vb_0(t_j,x_j)}^2 + \abs{\rb_0(t_j,x_j)}^2) \big) \abs{B_{r_j}(t_j,x_j)}.\label{Convergence-lemma-2}
\end{align}
Let us fix $k\in\N$, so that $\displaystyle\frac{1}{k} < \min\{r_0, \eps\}$, and choose a finite family of pairwise disjoint balls $B_{r_{k,j}}(t_{k,j},x_{k,j}) \subset \bO$ with radii $\displaystyle r_{k,j} < \frac{1}{k}$ such that \eqref{Convergence-lemma-2} holds true.

Thus, by applying the above construction in each ball $B_{r_{k,j}}(t_{k,j},x_{k,j})$, we obtain a sequence of smooth solutions $(\rb_{k,j}, \Eta_{k,j}, \vb_{k,j}, \bz_{k,j}, q_{k,j})$ of \eqref{DE}, which satisfy all the three properties, i.e., \eqref{p1}--\eqref{p3}. Thus, in particular, it follows that
\begin{equation*}
(\rb_k, \Eta_k, \vb_k, \bz_k, q_k) := (\rb_0, \Eta_0, \vb_0, \bz_0, q_0)+ \sum_j (\rb_{k,j}, \Eta_{k,j}, \vb_{k,j}, \bz_{k,j}, q_{k,j}) \in \X_0. 
\end{equation*}
In particular, from the property \eqref{p3}, and the estimates \eqref{Convergence-lemma-1} and  \eqref{Convergence-lemma-2}, we assert
\begin{align}
\int_{\bO} &(\abs{\vb_k(t,x) - \vb_0(t,x)}  + \abs{\rb_k(t,x) - \rb_0(t,x)}) \,\mathrm{d} x \,\mathrm{d} t\\
& = \sum_j \int_{B_{r_{k,j}}(t_{k,j},x_{k,j})} (\abs{\vb_{k,j}(t,x)} + \abs{\rb_{k,j}(t,x)}) \,\mathrm{d} x \,\mathrm{d} t\\
& \geq \alpha \sum_j \big( \abs{\overline{\vb}(t_{k,j},x_{k,j})} + |\overline{\rb}( t_{k,j},x_{k,j})| \big) |B_{r_{k,j}}(t_{k,j},x_{k,j})| \\
& \geq C\alpha \sum_j  ((h^2(t_{k,j})+1) - (\abs{\vb_0(t_{k,j},x_{k,j})}^2 + \abs{\rb_0(t_{k,j},x_{k,j})}^2)) |B_{r_{k,j}}(t_{k,j},x_{k,j})| \\
& \geq \frac{C\alpha}{2} \int_{\bO} ((h^2(t) +1) - (\abs{\vb_0(t,x)}^2 + \abs{\rb_0(t,x)}^2))\,\mathrm{d} x \,\mathrm{d} t. \label{Convergence-lemma-4}
\end{align}
Since $\int_{B_{r_{k,j}}(t_{k,j},x_{k,j})}(\rb_{k,j}, \Eta_{k,j}, \vb_{k,j}, \bz_{k,j}, q_{k,j})\,\mathrm{d} x \,\mathrm{d} t\to 0$ as $k\to\infty$, the above construction produces
a sequence $\{(\rb_k, \Eta_k, \vb_k, \bz_k, q_k)\}_{k\in\N}\in \X_0$ such that
\begin{align}\label{3p4}
(\rb_k, \Eta_k, \vb_k, \bz_k, q_k)\xrightharpoonup[k\to\infty]{\ast} (\rb_0, \Eta_0, \vb_0, \bz_0, q_0) \ \mbox{ in }\ \Lb^\infty(\bO).
\end{align}
\textbf{Step 4.} Furthermore, thanks to the estimate \eqref{Convergence-lemma-4} that yields
\begin{align*}
\left\| \vb_k\right\|_{\mL^2(\bO)}^2 + \left\| \rb_k\right\|_{\mL^2(\bO)}^2
& = \left\|\vb_0\right\|_{\mL^2(\bO)}^2 + \left\|\rb_0\right\|_{\mL^2(\bO)}^2
+\left\|\vb_k - \vb_0 \right\|_{\mL^2(\bO)}^2+\left\|\rb_k - \rb_0 \right\|_{\mL^2(\bO)}^2
\\&\quad+ 2 \left(\vb_0, \vb_k - \vb_0\right) +  2 \left(\rb_0, \rb_k - \rb_0\right)
\\
& \geq \left\|\vb_0\right\|_{\mL^2(\bO)}^2 + \left\|\rb_0\right\|_{\mL^2(\bO)}^2 + \frac{1}{|\bO|} \left\|\vb_k - \vb_0 \right\|_{\mL^1(\bO)}^2 + \frac{1}{|\bO|}  \left\|\rb_k - \rb_0 \right\|_{\mL^1(\bO)}^2\\
&\quad+ 2 \left(\vb_0, \vb_k - \vb_0\right) +  2 \left(\rb_0, \rb_k - \rb_0\right)\\
& \geq \left\|\vb_0\right\|_{\mL^2(\bO)}^2 + \left\|\rb_0\right\|_{\mL^2(\bO)}^2 + \frac{1}{2|\bO|} \big(\left\|\vb_k - \vb_0 \right\|_{\mL^1(\bO)} + \left\|\rb_k - \rb_0 \right\|_{\mL^1(\bO)}\big)^2\\ 
&\quad+ 2 \left(\vb_0, \vb_k - \vb_0\right) +  2 \left(\rb_0, \rb_k - \rb_0\right)\\
& \geq \left\|\vb_0\right\|_{\mL^2(\bO)}^2 + \left\|\rb_0\right\|_{\mL^2(\bO)}^2\\
&\quad + \frac{\alpha^2C^2}{8|\bO|} \left(\int_{\bO} ((h^2(t) +1) - (\abs{\vb_0}^2 + \abs{\rb_0(t,x)}^2))\mathrm{d} x \,\mathrm{d} t \right)^2\\ 
&\quad+ 2 \left(\vb_0, \vb_k - \vb_0\right) +  2 \left(\rb_0, \rb_k - \rb_0\right)\\
& = \left\|\vb_0\right\|_{\mL^2(\bO)}^2 + \left\|\rb_0\right\|_{\mL^2(\bO)}^2\\
&\quad + \frac{\alpha^2C^2}{8|\bO|} \left( |\bO_x|\int_0^Th^2(t)dt + |\bO| - (\left\|\vb_0\right\|_{\mL^2(\bO)}^2 + \left\|\rb_0\right\|_{\mL^2(\bO)}^2) \right)^2 \\
&\quad+ 2 \left(\vb_0, \vb_k - \vb_0\right) +  2 \left(\rb_0, \rb_k - \rb_0\right).
\end{align*}
Finally, by taking liminf on both sides in the above inequality and utilizing the convergence \eqref{3p4}, we infer
\begin{align*}
& \liminf_{k\to\infty}(\left\| \vb_k\right\|_{\mL^2(\bO)}^2 + \left\| \rb_k\right\|_{\mL^2(\bO)}^2)\\
& \geq \left\|\vb_0\right\|_{\mL^2(\bO)}^2 + \left\|\rb_0\right\|_{\mL^2(\bO)}^2 + \frac{\alpha^2C^2}{8|\bO|} \left( (|\bO_x| \|h\|_{L^2(0,T)}^2 +|\bO|) - (\left\|\vb_0\right\|_{\mL^2(\bO)}^2 + \left\|\rb_0\right\|_{\mL^2(\bO)}^2) \right)^2,
\end{align*}
where $\beta = \frac{\alpha^2C^2}{8|\bO|}$. It complete the proof.
\end{proof}

Now, we are in position to employ the Baire category method by using the geometric tools developed in the previous Subsection \ref{Sec-Geo setup}, which shows that there exists infinitely many solutions to the problem \eqref{RDDEE}, that eventually proves Theorem \ref{Main-result-Random-PDE}.

Let us begin by fixing a metric $d_{\infty}^{\ast}$ inducing the $w^\ast$ topology of $\Lb^\infty$ in $\X$, so that $(\X, d_{\infty}^{\ast})$ is a complete metric space.

\begin{definition}[{\cite[Definition 4.5]{LSJ-13} or \cite[p.\;57]{SL-88}}]
In a metric space $\X$, a map $J: \X \to \R$ is called a \emph{Baire-1} if it is a pointwise limit of continuous functions.
\end{definition}

\begin{definition}[{\cite[p.\;2]{JCO-80}}]
A set is said to be of \emph{first category} (or \emph{meagre}) if it can be represented as a countable union of nowhere dense sets.
\end{definition}

\begin{definition}[{\cite[p.\;41]{JCO-80} or \cite[p.\;14]{LSJ-13}}]
A topological space $\X$ is called a \emph{Baire space} if every non-empty open set in $\X$ is of second category, or equivalently; if the complement of every set of first category is dense. In a Baire space, the complement of any set of first category is called a \emph{residual} set.
\end{definition}

The following result is one of the novel results, motivated from \cite[Lemma 4.5]{CDL+LSJ-09}, of this work. It shows that the set of points of continuity is dense in $(\X, d_\infty^\ast)$.

\begin{lemma}\label{lem-Baire function}
The identity map
\begin{equation}\label{Def-I}
I : (\X, d_\infty^\ast) \to \Lb^2([0,\infty)\times \R^n)
\end{equation}
is a Baire-1 map and therefore the set of points of continuity is residual in $(\X, d_\infty^\ast)$.
\end{lemma}
\begin{proof}
Let us consider a regular space-time convolution kernel as
\[\phi_r(t,x) = r^{-(n+1)} \phi(rt,rx),\ r>0.\]
Let us fix $r>0$ and define a function 
$I_r: (\X, d_\infty^\ast) \to \Lb^2([0,\infty)\times \R^n)$  by
\begin{align*}
I_r(\rb, \Eta, \vb, \bz, q) := (\phi_r\ast \rb, \phi_r\ast \Eta, \phi_r\ast \vb, \phi_r\ast\bz, \phi_r\ast q)
\end{align*}
\textit{Claim:} $I_r$ is continuous. 

From Lemma \ref{Equivalent-lemma} note that the space $(\X, d_\infty^\ast)$ is compact and metrizable. Thus, it is sufficient to prove that  $I_r$ is sequentially continuous.
Since the convolution kernel $\phi_r\ast(\cdot)$ on $(\X, d_\infty^\ast)$ is a compact operator, for any $\X$-valued sequence  $\{(\rb_k, \Eta_k, \vb_k, \bz_k, q_k)\}_{k\in\N}$ such that  
\begin{align*}
(\rb_k, \Eta_k, \vb_k, \bz_k, q_k) \xrightharpoonup[k\to\infty]{} (\rb, \Eta, \vb, \bz, q) \ \ \mbox{  $w^\ast$ in  }\  \Lb^\infty([0,\infty)\times \R^n),
\end{align*}
by sequential compactness, it follows that
\begin{align*}
I_r(\rb_k, \Eta_k, \vb_k, \bz_k, q_k) 
\xrightarrow[k\to\infty]{} I_r(\rb, \Eta, \vb, \bz, q) \ \mbox{ in } \ \Lb^2([0,\infty)\times \R^n)).
\end{align*}
This proves that the map $I_r: (\X, d_\infty^\ast) \to \Lb^2([0,\infty)\times \R^n)$ is continuous.
Furthermore, 
\begin{align*}
I(\rb, \Eta, \vb, \bz, q)  = \lim_{r\to0} I_r(\rb, \Eta, \vb, \bz, q), \ \mbox{ for every }\  (\rb, \Eta, \vb,\bz, q) \in \X.
\end{align*}
Therefore, the identity map $I : (\X, d_\infty^\ast) \to \Lb^2([0,\infty)\times \R^n)$ is a pointwise limit of sequence of continuous maps, hence it is a Baire-1 map. Thus, it follows from \cite[Theorem 4.6]{LSJ-13}, its proof relies on \cite[Theorem 1.3]{JCO-80}, that the set of points of continuity of $I$ is residual in $(\X, d_\infty^\ast)$, which completes the proof.
\end{proof}

The next lemma provides a bridge between the points of continuity of $I$ and the weak solutions of the PDE \eqref{RDDEE-T} (as well as \eqref{RDDEE}) for almost every fixed path.

\begin{lemma}\label{Continuity-lemma}
If $(\rb, \Eta, \vb, \bz, q)\in \X$ is a point of continuity of $I$, then
\begin{equation}
\abs{\vb(t,x)} = h(t) \ \mbox{ and } \ \abs{\rb(t,x)} =1 \ \mbox{ for a.e. }\  (t,x) \in \bO.\label{Continuity-lemma-1}
\end{equation}
\end{lemma}

\begin{remark}\label{Rmk-equi}
Since for each $(\rb, \Eta, \vb, \bz, q)\in \X,$  it holds that
\begin{align*}
\abs{\vb(t,x)} \leq h(t)\ \mbox{ and }\ \abs{\rb(t,x)} \leq 1\ \mbox{ for a.e. }\ (t,x) \in \bO,
\end{align*}
Thus, \eqref{Continuity-lemma-1} is equivalent to 
\begin{equation}\label{Continuity-lemma-1-1}
\Vert\vb\Vert_{\mL^2(\bO)}^2 = \|h\|_{L^2(0,T)}^2 \abs{\bO_x}\ \mbox{ and }\ \Vert\rb\Vert_{\mL^2(\bO)}^2 = \abs{\bO}.
\end{equation}
\end{remark}

\begin{proof}[Proof of Lemma \ref{Continuity-lemma}]
Let us choose and fix $(\rb, \Eta, \vb, \bz, q)\in \X$, a point of continuity of $I$. Then, our aim is to show that \eqref{Continuity-lemma-1-1} holds.

\noindent
\textbf{Step 1.} By the definition of $\X$, $(\rb, \Eta, \vb,\bz,q)$ is a $w^\ast$ limit of a sequence in $\X_{0}$, i.e., there exists 
$\{(\rb_k ,\Eta_k, \vb_k,\bz_k,q_k)\}_{k\in\N}\subset\X_{0}$,  
such that
\begin{equation}\label{eqn-weak-conv-1}
(\rb_k, \Eta_k, \vb_k,\bz_k,q_k)\xrightharpoonup[k\to\infty]{} (\rb, \Eta, \vb,\bz,q)\ \  w^\ast\ \text{ in }\ \Lb^\infty(\bO).
\end{equation}
Since $(\rb, \Eta, \vb,\bz,q)$ is point of continuity of $I$, see \eqref{Def-I}, by sequential definition of continuity, it follows that
\begin{align}\label{Continuity-lemma-2}
(\rb_k, \Eta_k, \vb_k,\bz_k,q_k)\xrightarrow[k\to\infty]{}(\rb, \Eta, \vb,\bz,q) \ \mbox{ in }\  \Lb^2(\bO).
\end{align}
Next, by applying Lemma \ref{Convergence-lemma}, for each element $(\rb_k, \Eta_k, \vb_k,\bz_k,q_k)\in\X_{0}$, we find a subsequence 
\[ \big\{(\rb_{k,j}, \Eta_{k,j}, \vb_{k,j},\bz_{k,j},q_{k,j})\big\}_{j\in\N}\in\X_0,\]
such that
\begin{align}
(\rb_{k,j}, \Eta_{k,j}, \vb_{k,j},\bz_{k,j},q_{k,j}) \xrightharpoonup[j\to\infty]{} (\rb_k, \Eta_k, \vb_k,\bz_k,q_k)\ \  w^\ast\ \text{ in } \ \Lb^\infty(\bO),
\end{align}
and for some $\beta<\infty$, it satisfies
\begin{align}
&\hspace*{-2mm} \liminf_{j\to \infty}(\left\| \vb_{k,j}\right\|_{\mL^2(\bO)}^2 + \left\| \rb_{k,j}\right\|_{\mL^2(\bO)}^2)\\
&\hspace*{-2mm} \geq \left\|\vb_k\right\|_{\mL^2(\bO)}^2 + \left\|\rb_k\right\|_{\mL^2(\bO)}^2 + \beta \big( (\|h\|_{L^2(0,T)}^2|\bO_x| +|\bO|) - (\left\|\vb_k\right\|_{\mL^2(\bO)}^2 + \left\|\rb_k\right\|_{\mL^2(\bO)}^2) \big)^2.\label{step9-1}
\end{align}

\noindent
\textbf{Step 2.} Note from the definition of weak$^\ast$ convergence of a sequence that, there exists $j_0 = j_0(k) \in\N$ such that
\[d_\infty^\ast((\rb_{k,j}, \Eta_{k,j}, \vb_{k,j},\bz_{k,j},q_{k,j}), (\rb_{k}, \Eta_{k}, \vb_{k},\bz_{k},q_{k})) \leq \frac{1}{k}\  \text{ for every } \ j\geq j_0(k).\]
Thanks to the $w^\ast$ convergence \eqref{eqn-weak-conv-1} that yields
\begin{align*}
& d_\infty^\ast((\rb_{k,j_0}, \Eta_{k,j_0}, \vb_{k,j_0},\bz_{k,j_0},q_{k,j_0}), (\rb, \Eta, \vb,\bz,q))\\
& \leq	d_\infty^\ast((\rb_{k,j_0}, \Eta_{k,j_0}, \vb_{k,j_0},\bz_{k,j_0},q_{k,j_0}), (\rb_{k}, \Eta_{k}, \vb_{k},\bz_{k},q_{k}))\\
&\quad + d_\infty^\ast((\rb_{k}, \Eta_{k}, \vb_{k},\bz_{k},q_{k}), (\rb, \Eta, \vb,\bz,q))\\
& \leq
\frac{1}{k}  + d_\infty^\ast((\rb_{k}, \Eta_{k}, \vb_{k},\bz_{k},q_{k}), (\rb, \Eta, \vb,\bz,q)).
\end{align*}
Now, passing $k\to\infty$ and using the strong convergence \eqref{Continuity-lemma-2}
, we assert
\[(\rb_{k,j_0}, \Eta_{k,j_0}, \vb_{k,j_0},\bz_{k,j_0},q_{k,j_0}) \xrightharpoonup[k\to\infty]{} (\rb, \Eta, \vb,\bz,q)\ \  w^\ast\ \text{ in }\ \Lb^\infty(\bO).\]
But, again by the assumption, i.e., $(\rb, \Eta, \vb,\bz,q)$ is point of continuity of $I,$ it follows that
\begin{equation}\label{Continuity-lemma-3}
(\rb_{k,j_0}, \Eta_{k,j_0}, \vb_{k,j_0},\bz_{k,j_0},q_{k,j_0}) \xrightarrow[k\to\infty]{} (\rb, \Eta, \vb,\bz,q)\ \text{ in }\ \Lb^2(\bO).
\end{equation}
Therefore, inequality  \eqref{step9-1} implies that
\begin{align*}
\left\| \vb_{k,j_0}\right\|_{\mL^2(\bO)}^2 + \big\| \rb_{k,j_0}\big\|_{\mL^2(\bO)}^2
& \geq \liminf_{j\to \infty} \big(\left\| \vb_{k,j}\right\|_{\mL^2(\bO)}^2 + \left\| \rb_{k,j}\right\|_{\mL^2(\bO)}^2 \big)\\
& \geq \left\|\vb_k\right\|_{\mL^2(\bO)}^2 + \left\|\rb_k\right\|_{\mL^2(\bO)}^2 \\
&\quad + \beta \big( (\|h\|_{L^2(0,T)}^2 |\bO_x| +|\bO|) - \big(\left\|\vb_k\right\|_{\mL^2(\bO)}^2 + \left\|\rb_k\right\|_{\mL^2(\bO)}^2 \big) \big)^2.
\end{align*}
Thus, letting $k$ tends to $\infty$, and using the strong convergences \eqref{Continuity-lemma-2} and \eqref{Continuity-lemma-3} in the above inequality, we infer
\begin{align*}
0 &\geq \beta \big( (\|h\|_{L^2(0,T)}^2 |\bO_x| +|\bO|) - \big(\left\|\vb\right\|_{\mL^2(\bO)}^2 + \left\|\rb\right\|_{\mL^2(\bO)}^2\big) \big)^2,
\end{align*}
which further implies
\begin{align*}
0 = \big((\|h\|_{L^2(0,T)}^2 |\bO_x| +|\bO|) - \big(\left\|\vb\right\|_{\mL^2(\bO)}^2 + \left\|\rb\right\|_{\mL^2(\bO)}^2 \big)\big).
\end{align*}
Hence $\Vert\vb\Vert_{\mL^2(\bO)}^2 + \Vert\rb\Vert_{\mL^2(\bO)}^2 = \|h\|_{L^2(0,T)}^2 |\bO_x| +|\bO|$, but $\abs{\vb} \leq h(\cdot)$, $\abs{\rb} \leq 1$ in $\bO$ and the support of $\vb$ and $\rb$ is contained in $\bO$, it follows that
\begin{align*}
\abs{\vb} = h(\cdot)\mathds{1}_{\bO}\ \text{ and }\ \abs{\rb} = \mathds{1}_{\bO}  \iff \Vert\vb\Vert_{\mL^2(\bO)}^2  = \|h\|_{L^2(0,T)}^2 |\bO_x|\ \text{ and }\ \Vert\rb\Vert_{\mL^2(\bO)}^2 = \abs{\bO},
\end{align*}
which completes the proof.
\end{proof}

\begin{proof}[Proof of Theorem \ref{Main-result-Random-PDE}]
By Lemma \ref{Continuity-lemma}, every element of $\X$, which is also a point of continuity of $I$, see \eqref{Def-I}, satisfies
\begin{equation}\label{eqn-mode-1-prop}
\abs{\vb(t,x)} = h(t) \ \text{ and }\ \abs{\rb(t,x)} =1 \ \text{ for a.e. }\ (t,x) \in \bO.
\end{equation}
On the other hand, Lemma \ref{Equivalent-lemma} asserts that every element of $\X$ which satisfies the above property, i.e., \eqref{eqn-mode-1-prop}, solves the \eqref{RDDEE-T} (and \eqref{RDDEE}). Moreover, Lemma \ref{lem-Baire function} implies that the set of points of continuity of $I$ form a residual (dense) subset of $\X$. Hence, there exist infinitely many weak solutions to the PDE \eqref{RDDEE} for a.e. fixed $\omega \in \Omega$ , which completes the proof.
\end{proof}
\begin{remark}
Note that Theorem \ref{Main-result-Random-PDE} can be proved using the method of convex integration to produce infinitely many weak solutions of \eqref{RDDEE-T} as well as \eqref{RDDEE}, see Appendix \ref{CIT} for details.
\end{remark}

Since the random PDEs \eqref{RDDEE-T} and \eqref{RDDEE} admit infinitely many weak solutions, it suffices to map these solutions to corresponding solutions of the SPDE \eqref{DDEELMSN1}–\eqref{DDEELMSN4} in order to complete the proof of Theorem \ref{Main-result-Stochastic-PDE}. 


\subsection{Back to the SPDE}\label{Sec-Rescale}
In this subsection, we explain how to use a solution $\vb$ of the random PDEs \eqref{RDDEE-T} and  \eqref{RDDEE}, constructed via the Baire category technique in previous Subsection \ref{Sec-Baire category}, to recover a solution $\bu$ of the original system \eqref{DDEELMSN1}–\eqref{DDEELMSN2}. 
This can be done by the use of reverse of the transformations (formally justified), discussed in Subsection \ref{Sec-Transforming to RPDE}, to obtain the system \eqref{DDEELMSN1}-\eqref{DDEELMSN4}. 

First, let us recall a few abuses of notation used in the proof of Theorem \ref{Main-result-Random-PDE}, see Remarks \ref{Rmk-abuse-1} and \ref{Rmk-abuse-2}. 
In the case of transport noise, we denoted
\begin{equation}
	x = x + B(t),
\end{equation}
where $B$ is an $\R^n$-valued Brownian motion. 
Secondly, in the case of linear multiplicative noise, the time variable is in fact $\theta(t)$ rather than $t$, see \eqref{Time-transformation} for the definition of $\theta(\cdot)$.

\vspace{1mm}
\noindent
\textit{\underline{Transport noise}}. Let us choose and take $\vb$ to be a weak solution of the problem \eqref{RDDEE-T} guaranteed by Theorem \ref{Main-result-Random-PDE}. 
Then, we consider the following inverse transformation to get back to the SPDEs associated with transport noise:
\begin{equation*}
	\bu(t, x) =  \vb(t, x-B(t)).
\end{equation*}
Since the above transformation is simply a translation, i.e., a composition of $B(\cdot) \in \R^n$, a $\{\F_{t}\}_{t\geq 0}$-adapted process, with the following continuous function 
$$\int_{\R^n} u(t,x) \cdot \bvarphi (x) \d x = \int_{\R^n}  \vb(t,y) \bvarphi(y+\cdot) \d y,\ \text{ for all }\ \bvarphi \in \Cb^2(\R^n;\R^n),$$
where the continuity of the process $\int_{\R^n} \bu(t,x) \bvarphi (x) \d x $ follows from the weak continuity of $\vb(\cdot)$, see Remark \ref{Rmk-weak-c}. Hence the process in the left hand side is $\{\F_{t}\}_{t\geq 0}$-adapted. 
This allows us to recover that $\bu$ solves the original equations \eqref{DDEELMSN1} and \eqref{DDEELMSN2}, in the sense of Definition \ref{Weak-soln-T}. Consequently, we obtain infinitely many weak solutions of the random PDE corresponding to the SPDE \eqref{DDEELMSN1}–\eqref{DDEELMSN4} with transport noise.

\vspace{1mm}
\noindent
\textit{\underline{Linear multiplicative noise}.} Let us select and fix $\vb$ to be a weak solution to the problem \eqref{RDDEE} obtained in the previous section via Theorem \ref{Main-result-Random-PDE}. 
Then, consider the following time function
\begin{align}
[0,\infty) \ni \theta \mapsto \left\langle \vb, \bvarphi \right\rangle (\theta) \equiv \int_{\R^n}  \vb( \theta)\cdot \bvarphi\, \mathrm{d} x \in \R,\ \text{ for }\ \bvarphi \in \C_{0}^1 (\R^n; \R^n).
\end{align}
Observe that $\left\langle \vb, \bvarphi \right\rangle$ is a globally Lipschitz function in time, i.e., for all $\theta, \theta_0\in[0,\infty)$ (without loss of generality, let $\theta > \theta_0$), we have
\begin{align*}
&	\abs{\left\langle \vb, \bvarphi \right\rangle(\theta) - \left\langle \vb, \bvarphi \right\rangle(\theta_0)}
\\	& = \abs{ \int_{\R^n } \vb( \theta)\cdot \bvarphi \,\mathrm{d} x -  \int_{\R^n }\vb( \theta_0)\cdot \bvarphi \,\mathrm{d}x}\\
& =\bigg| \int_{\R^n } \vb( 0) \cdot\bvarphi\,\mathrm{d} x + \int_0^\theta \int_{\R^n} (\vb\otimes \vb): \nabla \bvarphi \,\mathrm{d}x \,\mathrm{d}\tau+ \int_0^\theta \int_{\R^n}p\, \di\ \bvarphi \,\mathrm{d}x \,\mathrm{d}\tau \\
& \qquad - \int_{\R^n } \vb( 0) \cdot\bvarphi\,\mathrm{d} x + \int_0^{\theta_0} \int_{\R^n} (\vb\otimes \vb): \nabla \bvarphi \,\mathrm{d}x \,\mathrm{d}\tau +\int_0^{\theta_0} \int_{\R^n}p\, \di\ \bvarphi \,\mathrm{d}x \,\mathrm{d}\tau\bigg|\\
& \leq M\left[ \int_{\theta_0}^\theta \int_{\R^n} \abs{(\vb\otimes \vb)} \,\mathrm{d}x \,\mathrm{d}\tau+\int_{\theta_0}^\theta \int_{\R^n} \abs{p} \,\mathrm{d}x \,\mathrm{d}\tau\right]\\
& \leq  \wtilde{M} \abs{\theta - \theta_0},
\end{align*}
where we have used  Definition \ref{Def-weak-rand} (weak formulation).
By an application of Rademacher's Theorem, see \cite[Theorem 2.2.1]{WPZ-12}, and abstract chain rule, see \cite{WPZ-12}, it  follows that for a.e. $t\geq 0$,
\begin{align*}
\frac{\mathrm{d}}{\mathrm{d}t}\left\langle \vb, \bvarphi \right\rangle
= \partial_t\int_{\R^n} \vb\Big( \int_{0}^{t} e^{-\gamma B(s)}\,\mathrm{d}s\Big)\cdot \bvarphi \,\mathrm{d} x = e^{-\gamma B(t)} \int_{\R^n } \vb_t( t) \cdot \bvarphi \,\mathrm{d} x.
\end{align*}
Once again using the weak formulation, we obtain
\begin{align}
\frac{\mathrm{d}}{\mathrm{d} t} \left\langle \vb, \bvarphi \right\rangle
& = e^{-\gamma B(t)} \int_{\R^n } [\vb( t)\otimes \vb ( t) : \nabla \bvarphi + p(t)\, \mathrm{div \, } \bvarphi]\, \mathrm{d} x  \label{Transformation-eqn-1}
\end{align}
and
\begin{align*}
\int_{\R^n } \vb ( t) \cdot \nabla\bvarphi \,\mathrm{d} x = 0.
\end{align*}
Now, using the basic properties of Itô's integral, i.e., if $b$ is Lipschitz, then it follows that
\begin{align}\label{Transformation-eqn-2}
\mathrm{d} (b e^{-\gamma B(t)}) = e^{-\gamma B(t)} \,\mathrm{d} b - \gamma b e^{-\gamma B(t)} \,\mathrm{d} B(t) + \frac{\gamma^2}{2}b e^{-\gamma B(t)} \,\mathrm{d}t.
\end{align}
By taking $b = \left\langle \vb, \bvarphi \right\rangle $ in \eqref{Transformation-eqn-2} and defining the required velocity 
\begin{equation}
\bu(t) := e^{-\gamma B(t)} \vb(t) \in \mathrm{C}_{w}([0,\infty); \mL^2(\R^n; \R\times \R^n)),
\end{equation}
we obtain 
\begin{align}
\mathrm{d} \int_{\R^n }  \vb(t) \cdot \bvarphi \,\mathrm{d} x 
& =  e^{\gamma B(t)}\bigg[\mathrm{d} \int_{\R^n } \bu(t) \cdot \bvarphi \,\mathrm{d} x + \gamma\int_{\R^n } \bu(t) \cdot \bvarphi \,\mathrm{d} x \,\mathrm{d} B(t)
\\
&\qquad\qquad - \frac{\gamma^2}{2}  \int_{\R^n} \bu(t) \cdot \bvarphi \,\mathrm{d} x \,\mathrm{d} t\bigg].\label{Transformation-eqn-3}
\end{align}
Then, comparing the RHS of \eqref{Transformation-eqn-1} and \eqref{Transformation-eqn-3}, we assert
\begin{align*}
e^{\gamma B(t)}\bigg[\mathrm{d} & \int_{\R^n } \bu(t) \cdot \bvarphi \,\mathrm{d} x + \gamma\int_{\R^n } \bu(t) \cdot \bvarphi \,\mathrm{d} x \,\mathrm{d} B(t) - \frac{\gamma^2}{2}  \int_{\R^n} \bu(t) \cdot \bvarphi \,\mathrm{d} x \,\mathrm{d} t\bigg]\\ 
& = e^{-\gamma B(t)} \bigg[\int_{\R^n } \vb(t)\otimes \vb(t) : \nabla \bvarphi \,\mathrm{d} x\,\mathrm{d} t + \int_{\R^n } p(t)\, \mathrm{div \, } \bvarphi\, \,\mathrm{d} x \,\mathrm{d} t \bigg].
\end{align*}
It finally implies 
\begin{align*}
\mathrm{d} \int_{\R^n } \bu(t) \cdot \bvarphi \,\mathrm{d} x 
&\,  =  \int_{\R^n } \bu(t)\otimes \bu(t) : \nabla \bvarphi \,\mathrm{d} x \,\mathrm{d} t + \int_{\R^n } \pi(t)\, \mathrm{div \, } \bvarphi \,\mathrm{d} x \,\mathrm{d} t  \\
&\quad  + \frac{\gamma^2}{2}  \int_{\R^n} \bu(t) \cdot \bvarphi \,\mathrm{d} x \,\mathrm{d} t - \gamma  \int_{\R^n} \bu(t) \cdot \bvarphi \,\mathrm{d} x \,\mathrm{d} B(t)\ \mbox{ $\Pr$-a.s.}
\end{align*}
Thus, integrating both sides with respect to time yields \eqref{eqn-weak-form}. Observe that, due to the $\F_t$-measurability of the process $\theta(t)$, the process $\bu$ is $\{\F_{t}\}_{t\geq0}$-adapted for any fixed $\bvarphi \in C^1(\R^n ;\R^n)$, where the filtration $\{\F_{t}\}_{t\ge0}$ corresponds to the noise $B(t)$.
Similarly, the weak formulation for the tracer $\rb$ can be recovered by using the reverse of the  transformations discussed above. 

Hence, it follows that there exists infinitely many weak solutions (in the sense of Definition \ref{Weak-soln}) of the SPDE \eqref{DDEELMSN1}--\eqref{DDEELMSN4}, which conclude the proofs of Theorem \ref{Main-result-Stochastic-PDE}.

\appendix
\renewcommand{\thesection}{\Alph{section}}
\numberwithin{equation}{section}
\section{Application and alternative method}

\subsection{An application to stochastic ideal MHD equations}\label{ASMHD}
One of the main motivations for considering the incompressible stochastic Euler equations with a passive tracer driven by Stratonovich-type noise is their application to the 3D stochastic ideal MHD equations perturbed by both transport and linear multiplicative Stratonovich noise (see \cite[Section 4]{CAB+CMFL+JHLN-15} for the deterministic case in higher dimensions; see also \cite{MC+ZZ-25}).

Consider the 3D stochastic ideal MHD equations
\begin{equation}\label{SMHDE}
\left\{\begin{aligned}
& \bu_t(t) + (\bu(t)\cdot\nabla)\bu(t) - (\curl\, \bb(t))\times \bb(t) + \nabla \pi(t) = - \Nn(\bu(t)) \circ \,\mathrm{d} B(t),\\
&  \bb_t(t) - \curl(\bu(t)\times \bb(t)) = \0,\\
& \mathrm{div \, } \bu(t) = 0,\\
& \mathrm{div \, } \bb(t) = 0,\\
& \bu(0) = \0,\ \rb(0) = 0,
\end{aligned}\right.
\end{equation}
where $\bu(t) = \bu(t,x) \in \R^{3}$ is the velocity field, $\bb(t) = \bb(t,x)\in \R^3$ is the magnetic induction, $\pi(t)= \pi(t,x) \in \R$ is the pressure, with $(t,x)\in[0,\infty)\times \R^3$. Furthermore, $\Nn(\cdot)$ is defined as in \eqref{eqn-noise}, which covers both transport and linear multiplicative noise. The symbol $\circ$ indicates that the stochastic integral is understood in the Stratonovich sense, and $B(\cdot)$ is a Brownian motion as defined in \eqref{eqn-def-B}.

Observe that the following restrictions on velocity and magnetic induction
\begin{align*}
\bu(t,x)  = (u_1(t,x), u_2(t,x), 0),\
\bb(t,x) = (0, 0, \rb(t,x)),\
(t,x)  = (t,x_1,x_2) \in [0,\infty)\times \R^2,
\end{align*}
leads to the following known system:
\begin{equation}\label{TMHDE}
\left\{\begin{aligned}
& \bu_t(t) + \mathrm{div \, }(\bu(t)\otimes \bu(t) ) + \nabla p(t) = - \Nn(\bu(t)) \circ \,\mathrm{d} B(t),\\
& \rb_t(t) + \mathrm{div \, }(\rb(t) \bu(t)) = 0,\\
& \mathrm{div \, } \bu(t) = 0\\
& \bu(0) = \0,\ \rb(0) = 0.
\end{aligned}\right.
\end{equation}
with
\begin{equation*}
p(t) := \pi(t) + \frac{\abs{\rb(t)}^2}{2}.
\end{equation*}
Thus, the problem \eqref{TMHDE} reduces to the stochastic Euler equations with a passive tracer, i.e., \eqref{DDEELMSN1}–\eqref{DDEELMSN4}. Hence, by applying Theorem \ref{Main-result-Random-PDE} with $n=2$, we obtain non-uniqueness of weak solutions to the stochastic ideal MHD equations \eqref{SMHDE}.

\dela{
Note that, by the transformations discussed in Subsection \ref{Sec-Transforming to RPDE}, we can transform the above mentioned system into the following random system:
\begin{equation}\label{RMHDE}
	\left\{\begin{aligned}
		& \vb_t(\theta) + (\vb(\theta)\cdot\nabla)\vb(\theta) - (\curl\, \bb(\theta))\times \bb(\theta) + \nabla \wtilde{p}(\theta) = - \Nn(\bu(t)) \circ \,\mathrm{d} B(t),\\
		&  \bb_\theta(\theta) - \curl(\vb(\theta)\times \bb(\theta)) = \0,\\
		& \mathrm{div \, } \vb(\theta) = 0,\\
		& \mathrm{div \, } \bb(\theta) = 0,\\
		& \vb(0) = \0,\ \rb(0) = 0,
	\end{aligned}\right.
\end{equation}
where $\theta$ is defined in \eqref{Time-transformation}.
}

Next, for the sake of completeness, we provide an alternative proof of the main result for the random PDEs \eqref{RDDEE-T} and \eqref{RDDEE}, as follows.

\subsection{Convex integration proof}\label{CIT}
For the reader’s convenience, we include a proof of Theorem~\ref{Main-result-Random-PDE} via convex integration for $n\geq 2$. This argument is a straightforward generalization of \cite[Theorem~1.1]{CAB+CMFL+JHLN-15}, which treats the case $n=2$ with constant energy equal to $1$, here we extend it to any continuous energy profile bounded away from zero. For simplicity, we adopt the notation introduced in Subsection~\ref{Sec-Baire category}.

\begin{proof}[Proof of Theorem \ref{Main-result-Random-PDE}]
The theme of the proof is to construct a sequence $\{\omegaa_k = (\rb_k, \Eta_k, \vb_k,\\ \bz_k, q_k)\}_{k\in\N}\subset \X_0$ satisfying the following two conditions:
\begin{itemize}
\item[(i)] there exists $\omegaa = (\rb, \Eta, \vb, \bz, q)\in\X$ such that $\omegaa_k  \xrightarrow[k\to \infty]{} \omegaa$  in $\Lb^2([0,\infty)\times \R^n)$;
\item[(ii)] for all $k\in\N$,
\end{itemize}
\begin{align*}
\left\| \vb_{k+1}\right\|_{\mL^2(\bO)}^2 + \left\| \rb_{k+1}\right\|_{\mL^2(\bO)}^2 & \geq \left\|\vb_k\right\|_{\mL^2(\bO)}^2 + \left\|\rb_k\right\|_{\mL^2(\bO)}^2 \\
& \quad + \beta \big( (\|h\|_{L^2(0,T)}^2|\bO_x| + |\bO|) - (\left\|\vb_k\right\|_{\mL^2(\bO)}^2 + \left\|\rb_k\right\|_{\mL^2(\bO)}^2) \big)^2.
\end{align*}
Then, by using (i), we can pass to the limit in (ii) to obtain
\begin{align*}
&\left\| \vb\right\|_{\mL^2(\bO)}^2 + \left\| \rb\right\|_{\mL^2(\bO)}^2\\
& \geq \left\|\vb\right\|_{\mL^2(\bO)}^2 + \left\|\rb\right\|_{\mL^2(\bO)}^2 + \beta \big( (\|h\|_{L^2(0,T)}^2|\bO_x| +|\bO|) - (\left\|\vb\right\|_{\mL^2(\bO)}^2 + \left\|\rb\right\|_{\mL^2(\bO)}^2) \big)^2,
\end{align*}
and hence $\left\| \vb\right\|_{\mL^2(\bO)}^2 + \left\| \rb\right\|_{\mL^2(\bO)}^2  = (\|h\|_{L^2(0,T)}^2|\bO_x| +|\bO|)$. Since $\abs{\vb} \leq h(t)$, $\abs{\rb} \leq 1$ in $\bO$, and supported in $\bO$, we conclude that $\abs{\vb}= h(t)\mathds{1}_{\bO}$ and $\abs{\rb} =\mathds{1}_{\bO}$. Clearly, $(\rb, \Eta, \vb, \bz)\in K_t^{\mathrm{co}}$ for a.e. $(t,x) \in \bO$, since $(\rb, \Eta, \vb, \bz, q)\in\X$. This implies that $(\rb, \Eta, \vb, \bz)(t,x)\in K_t$ for a.e. $(t,x)\in\bO$, as we needed.

It remains to construct a sequence $\{\omegaa_k\}_{k\in\N} \subset\X_0$ satisfying (i) and (ii). In order to construct, we set $(\rb_1, \Eta_1, \vb_1, \bz_1, q_1) = \0$ in $[0,\infty)\times \R^n$. This is possible due to Lemma \ref{Lem-U}.
For fixed $\eps>0$, let $\rho_\eps$ be a standard mollifying kernel in $[0,\infty)\times \R^n)$. The sequence $(\rb_k, \Eta_k, \vb_k,\bz_k, q_k) \in \X_0$ is constructed inductively, as well as an auxiliary sequence of numbers
\begin{align}\label{CI_Proof_eqn-1}
\delta_k < 2^{-k},
\end{align}
such that
\begin{align}\label{CI_Proof_eqn-2}
\left\|\omegaa_k - \omegaa_k \ast \rho_{\delta_k} \right\|_{\mL^2(\bO)} < 2^{-k}.
\end{align}
Then, we apply Lemma \ref{Convergence-lemma} to obtain $\omegaa_{k+1} = (\rb_{k+1}, \Eta_{k+1}, \vb_{k+1},\bz_{k+1}, q_{k+1}) \in \X_0$ such that
\begin{align}
&\Vert \vb_{k+1} \Vert_{\mL^2(\bO)}^2 + \Vert \rb_{k+1} \Vert_{\mL^2(\bO)}^2\\ 
& \geq \Vert \vb_k \Vert_{\mL^2(\bO)}^2 + \Vert \rb_k \Vert_{\mL^2(\bO)}^2 + \beta \big( (\|h\|_{L^2(0,T)}^2|\bO_x| +|\bO|) - (\Vert \vb_k \Vert_{\mL^2(\bO)}^2 + \Vert \rb_k \Vert_{\mL^2(\bO)}^2)\big)^2,\label{CI_Proof_eqn-3}
\end{align}
and
\begin{align}\label{CI_Proof_eqn-4}
\left\|(\omegaa_{k+1} - \omegaa_k) \ast \rho_{\delta_j} \right\|_{\mL^2(\bO)} < 2^{-k}\ \text{ for all }\ j \leq k.
\end{align}
Since the sequence $\{\omegaa_k\}_{k\in\N}$ is bounded in $\Lb^\infty([0,\infty)\times \R^n)$, there exists a subsequence, which we still denote by $\omegaa_k$, and a vector field $\omegaa = (\rb, \Eta, \vb, \bz, q) \in \X$ such that $\omegaa_k \xrightharpoonup[k\to\infty]{} \omegaa$\ $\ w^\ast$ in $\Lb^\infty([0,\infty)\times \R^n)$. Moreover, the sequence $\{\omegaa_k\}_{k\in\N}$ and the corresponding sequence $\{\delta_k\}_{k\in\N}$ satisfy the properties \eqref{CI_Proof_eqn-1}, \eqref{CI_Proof_eqn-2},  \eqref{CI_Proof_eqn-3} and \eqref{CI_Proof_eqn-4}. Then, for every $k\in\N$
\begin{equation*}
\left\|\omegaa_k\ast\rho_{\delta_k} - \omegaa\ast \rho_{\delta_k} \right\|_{\mL^2(\bO)} \leq \sum_{j = 0}^{\infty} \left\|\omegaa_{k+j}\ast\rho_{\delta_k} - \omegaa_{k+j+1}\ast \rho_{\delta_k} \right\|_{\mL^2(\bO)}
\leq \sum_{j = 0}^{\infty} 2^{-k+j} \leq 2^{-k+1}.
\end{equation*}
By using the triangle inequality, we assert
\begin{align*}
\left\|\omegaa_k - \omegaa \right\|_{\mL^2(\bO)} \leq \left\|\omegaa_k - \omegaa_k\ast \rho_{\delta_k} \right\|_{\mL^2(\bO)} + \left\|\omegaa_k\ast\rho_{\delta_k} - \omegaa\ast \rho_{\delta_k} \right\|_{\mL^2(\bO)} + \left\|\omegaa\ast\rho_{\delta_k} - \omegaa \right\|_{\mL^2(\bO)},
\end{align*}
and letting $k\to \infty$, we deduce that $\omegaa_k  \xrightarrow[k\to \infty]{} \omegaa$  in $\Lb^2(\bO)$. This concludes the proof of Theorem \ref{Main-result-Random-PDE}
\end{proof}


Lastly, to complete this manuscript, we provide the proof of an essential result that ensures the non-emptiness of the open set $\U_t$, as follows.

\subsection{Proof of Lemma \ref{Lem-U}}\label{Subsec-U}
Here we present the proof of Lemma \ref{Lem-U}, motivated from \cite[Lemma 4.2]{CDL+LSJ-09}. The main idea of the proof is that, for a.e. fixed $t \in [0,T]$, we define an operator $\T_t$ (see \eqref{Def-T}) and, by exploiting the symmetry of the sphere of radius $h(t)$ together with the orthogonality properties of the surface (Haar) measure $\nu_t$ (see below), establish that $\T_t$ is surjective.

\begin{remark}\label{Symmetricity}
Throughout this section, let $\sphere_t^{n-1}$ denote the sphere of radius $h(t)$ in $\R^n$ centered at the origin. The symmetricity of $\sphere_t^{n-1}$ means that, for any $v\in \sphere_t^{n-1}$, there exists $O\in O(n)$ (the space of $n\times n$ orthogonal matrices) such that $Ov = O (v_1, \cdots, v_i, \cdots, v_n) = O (v_1, \cdots, -v_i, \cdots, v_n)$. Since the Haar measure $\nu_t$ is rotational invariant, it follows that
\begin{equation*}
\int_{\sphere_t^{n-1}}v_i\, d\nu_t =\int_{\sphere_t^{n-1}} -v_i\, d\nu_t = 0,\ \text{ for }\ 1 \le i \le n.
\end{equation*} 
\end{remark}

\begin{lemma}[{\cite[Lemma 2]{CDL+LSJ-10}}]\label{Lem-e-con}
Let us consider a mapping
\begin{equation}
\R^n\times M_0^n  \ni (\vb, \bu) \mapsto e(\vb,\bu) := \frac{n}{2}\lambda_{\max} \left(\vb\otimes \vb - \bu\right) \in \R,
\end{equation}
where $\lambda_{\max}$ denotes the largest eigenvalue. Then, the following hold:
\begin{itemize}
\item[(i)] $e: \R^n\times M_0^n \to \R$ is convex,
\item[(ii)] $\frac{|\vb|^2}{2} \le e(\vb, \bz)$ and equality if and only if $\bz = \vb \otimes \vb - \frac{|\vb|^2}{2} I_n$,
\item[(iii)] $\|\bz\|_{\mathrm{op}} \le \frac{2(n-1)}{n} e(\vb, \bz)$, where $\|\cdot\|_{\mathrm{op}}$ denotes the operator norm of the matrix,
\item[(iv)] If $t \geq 0$, then the $\frac{h^2(t)}{2}$-sublevel set of $e$ is the convex hull of $\wtilde{K}_t$, i.e.,
\begin{equation}\label{Def-K_t-e}
\bigg\{ (\wtilde{\vb}, \wtilde{\bz}) \in 
\R^n\times\matr_0^n : e(\wtilde{\vb},\wtilde{\bz}) \le \frac{h^2(t)}{2}
\bigg\}=\wtilde{K}_t^{\mathrm{co}},
\end{equation}
where $\wtilde{K}_t$ is a restriction set consisting of pairs$(\wtilde{\vb},\wtilde{\bz})$ from the set $K_t$, see \eqref{K}.
\end{itemize} 
\end{lemma}

\begin{proof}[Proof of Lemma \ref{Lem-U}]
Let us fix a.e.~$t\ge0$ and recall from \eqref{eqn-U} that $\U_t$ is the set of interior points of the set ${\K}_t^{co}$, see \eqref{K} for definition of ${\K}_t$. Since $\wtilde{\rb} \in [-1,1]$, we have $0\in(-1,1)$. Therefore, it is sufficient to prove that the element $\coma{0} = (\0,\adda{0})\in\R^n\times M_0^n$ belongs to $\wtilde{\U}_t := \mathrm{int}\wtilde{K}_t^{\mathrm{co}}$,
which in turn implies that $\wtilde{\U}_t$ is non-empty.

Let $\nu_t$ be a surface 
measure on $\sphere_t^{n-1}$, see Remark \ref{Symmetricity}, such that $\nu_t(\sphere_t^{n-1})=1$ and consider the linear map
\begin{equation}\label{Def-T}
\Cb(\sphere_t^{n-1})\ni \phi \mapsto \T_t(\phi):=\int_{\sphere_t^{n-1}} \left(\vb,\vb\otimes \vb - \frac{h^2(t)}{n}I_n \right)\phi(\vb)\, d\nu_t \in  \R^n\times M_0^n.
\end{equation}
\textbf{Step 1.} The map  $\T_t(\cdot) := (\T_t^1(\cdot), \T_t^2(\cdot))$, is well-defined and linear, where
\begin{equation*}
\Cb(\sphere_t^{n-1}) \ni \phi \mapsto \T_t^1(\phi)  = \int_{\sphere_t^{n-1}}\vb\phi(\vb)\, d\nu_t \in \R^n,
\end{equation*}
and \begin{align*}
\Cb(\sphere_t^{n-1})\ni \phi \mapsto \T_t^2(\phi)  := \int_{\sphere_t^{n-1}}\left(\vb\otimes \vb - \frac{h^2(t)}{n} I_n\right)\phi(\vb)\,d\nu_t \in M_0^n.
\end{align*}
To do so, it is enough to show $\T_t^2(\phi)\in M_0^n$. Let us fix $\vb \in \sphere_t^{n-1}$. Then, it follows that
\begin{align*}
\Tr(\T_t^2(\phi)) & = \Tr\left( \int_{\sphere_t^{n-1}} \left(\vb\otimes \vb - \frac{h^2(t)}{n} I_n \right)\phi(\vb)\, d\nu_t\right)  =  \sum_{i=1}^{n} \int_{\sphere_t^{n-1}}\left(v_i^2 -\frac{h^2(t)}{n}\right)\phi(\vb)\, d\nu_t\\
& =  \int_{\sphere_t^{n-1}}(|\vb|^2 - h^2(t))\phi(\vb)\, d\nu_t = 0.
\end{align*}
Hence $\T_t(\phi)\in\R^n\times M_0^n$ and the linearity is trivial.
\vskip 1mm
\noindent
\textbf{Step 2.} Our next aim is to show that $\T_t(\phi)\in\wtilde{\K}_t^{co}$ for
\begin{align}\label{EEqP10}
0 \le \phi \in \Cb(\sphere_t^{n-1})\ \text{ and }\ \int_{\sphere_t^{n-1}} \phi\, d\nu_t = 1.
\end{align}
Suppose if $\phi\in\Cb(\sphere_t^{n-1})$ is having the density properties \eqref{EEqP10}.
Since the mapping $e$ is a convex function, see Lemma \ref{Lem-e-con}, it follows from Jensen's inequality that 
\begin{align*}
e\left( \int_{\sphere_t^{n-1}}\left(\vb, \vb\otimes \vb - \frac{h^2(t)}{n}I_n\right)\phi(\vb)\,d\nu_t \right) & \leq \int_{\sphere_t^{n-1}} e\left(\vb, \vb\otimes \vb - \frac{h^2(t)}{n}I_n\right) \phi(\vb)\, d\nu_t\\
& \leq \int_{\sphere_t^{n-1}} \frac{h^2(t)}{2} \phi(\vb)\, d\nu_t\\
& = \frac{h^2(t)}{2}.
\end{align*}
Thus, by \eqref{Def-K_t-e}, we obtain $\T_t(\phi)\in\wtilde{\K}_t^{co}.$ 
\vskip 1mm
\noindent
\normalsize 
\textbf{Step 3.} Now, for $\phi\equiv1$, we show that  $\displaystyle \T_t(1) = \coma{0}\in\wtilde{\K}_t^{co}.$
First note that the constant function $\phi\equiv 1$ satisfies \eqref{EEqP10}, which implies $\T_t(1)\in\wtilde{\K}_t^{co}$. Further, from the definition of $\T_t$, it follows that
\begin{align*}
\T_t(1) = \left(\int_{\sphere_t^{n-1}}\vb\, d\nu_t, \int_{\sphere_t^{n-1}}\left(\vb\otimes \vb - \frac{h^2(t)}{n}I_n\right)\,d\nu_t\right).
\end{align*}
Since the surface measure is rotational invariant, 
if we denote 
\[\ab_t =\int_{\sphere_t^{n-1}}\vb\, d\nu_t\in\R^n,\]
then for any $O\in O(n)$, it implies
\[O\ab_t= O\left(\int_{\sphere_t^{n-1}}\vb\, d\nu_t(\vb)\right) = \int_{\sphere_t^{n-1}}O\vb\, d\nu_t(\vb) = \int_{\sphere_t^{n-1}}\wb\, d\nu_t(\wb) = \ab_t,\]
where $O(n)$ denotes the space of $n\times n$ orthogonal matrices and we used the  transformation $O\vb=w$ that also implies $\nu_t(\vb) = \nu_t(O^{\top}\wb)= \nu_t(\wb)$. If $(\0\neq)\ab_t\in\R^n$, then there exists an $\wtilde{O}$ (for e.g. take $\wtilde{O}=-I$) so that $\wtilde{O}\ab_t=-\ab_t $. Therefore, it yields 
\[\ab_t=\int_{\sphere_t^{n-1}}\vb\, d\nu_t = \0.\]
Similarly, suppose $\displaystyle A=\int_{\sphere_t^{n-1}}\left(\vb\otimes \vb - \frac{h^2(t)}{n}I_n\right)\,d\nu_t\in M_0^n$ is such that 
\begin{align*}
OAO^{\top} & = \int_{\sphere_t^{n-1}}\left(O(\vb\otimes \vb)O^{\top} - \frac{h^2(t)}{n}I_n\right)\,d\nu_t(\vb) = \int_{\sphere_t^{n-1}}\left(O\vb^{\top}\vb O^{\top} - \frac{h^2(t)}{n}I_n\right)\,d\nu_t(\vb)\\
& = \int_{\sphere_t^{n-1}}\left(\left(\vb O^{\top} \right)^{\top} \vb O^{\top} - \frac{h^2(t)}{n}I_n\right)\,d\nu_t(\vb)\\ 
& = \int_{\sphere_t^{n-1}}\left(\left(\vb O^{\top}\right)\otimes \left(\vb O^{\top}\right) - \frac{h^2(t)}{n}I_n\right)\,d\nu_t(\vb)\\
& = \int_{\sphere_t^{n-1}}\left(\wb\otimes \wb - \frac{h^2(t)}{n}I_n\right)\,d\nu_t(\wb)\\
& = A,
\end{align*}
where we used the  transformation $\vb O^{\top} = \wb$ which implies $\nu_t(\vb) = \nu_t(\wb O)= \nu_t(\wb)$. Then, there exists a diagonal matrix $D$ such that 
\[ OAO^{\top} = D = A.\]
Since the above identity holds for any $O\in O(n)$, we may in particular take $O=O_{per}$, where $O_{per}$ is the permutations of the identity matrix $I_n$. For example, in case of $n=2$, if we consider 
\[O_{per} = \begin{bmatrix}
0 & 1\\ 1 & 0
\end{bmatrix}.\]
Then, we calculate to get
\[O_{per} A O_{per}^{\top} = O_{per} D O_{per}^{\top} = \begin{bmatrix}
0 & 1\\ 1 & 0
\end{bmatrix} \begin{bmatrix}
\lambda_1 & 0\\ 0 & \lambda_2
\end{bmatrix} \begin{bmatrix}
0 & 1\\ 1 & 0
\end{bmatrix} = \begin{bmatrix}
\lambda_2 & 0\\ 0 & \lambda_1
\end{bmatrix} . \] 
On the other hand 
\[
A=  \begin{bmatrix}
\lambda_1 & 0\\ 0 & \lambda_2
\end{bmatrix}.
\]
Hence it  follows that $\lambda_1 = \lambda_2$. Therefore $A = \lambda I_n$, for some $\lambda\in\R$. On the other hand, for $\vb\in\sphere_t^{n-1}$, $\Tr(\vb\otimes \vb) = h^2(t)$ that implies \[\Tr A=0 \implies \underbrace{\lambda + \cdots + \lambda}_{n-times} = n\lambda = 0,\]
which gives $A= \adda{0}$, i.e., $\T_t(1) = \coma{0}.$
Hence $\coma{0}\in\wtilde{\K}_t^{co}$.
\vskip 1mm
\noindent
\textbf{Step 4.} Observe that whenever $\psi\in\Cb(\sphere_t^{n-1})$ such that
\begin{align}\label{EEqP11}
0\leq\Vert \psi\Vert_{\Cb(\sphere_t^{n-1})} \leq 1 - \int_{\sphere_t^{n-1}} \psi\, d\nu_t = \wtilde{\alpha}, \ \text{ for some } \ \wtilde{\alpha}\in\R.
\end{align}
Then, every $\phi \in \Cb(\sphere_t^{n-1})$ of the form $\phi = \wtilde{\alpha} + \psi$ will satisfies \eqref{EEqP10}, i.e.,
\begin{align*}
\phi = \wtilde{\alpha} + \psi 
\geq \Vert \psi\Vert_{\Cb(\sphere_t^{n-1})} + \psi \geq \Vert \psi\Vert_{\Cb(\sphere_t^{n-1})} - \Vert \psi\Vert_{\Cb(\sphere_t^{n-1})} = 0,
\end{align*}
where we used the fact $\abs{\psi} \leq \Vert \psi\Vert_{\Cb(\sphere_t^{n-1})}$ implies $-\Vert \psi\Vert_{\Cb(\sphere_t^{n-1})} \leq \psi \leq \Vert \psi\Vert_{\Cb(\sphere_t^{n-1})}$ and 
\begin{align*}
\int_{\sphere_t^{n-1}} \phi\, d\nu_t = \int_{\sphere_t^{n-1}}(\wtilde{\alpha} + \psi )\, d\nu_t = \wtilde{\alpha} \int_{\sphere_t^{n-1}}1\, d\nu_t  + \int_{\sphere_t^{n-1}}\psi \, d\nu_t = 1.
\end{align*}
Thus, using Step 3, it follows that
\begin{align*}
\T_t(\phi) & = \int_{\sphere_t^{n-1}}\left(\vb\otimes \vb - \frac{h^2(t)}{n}I_n\right)\phi\,d\nu_t  = \int_{\sphere_t^{n-1}}\left(\vb\otimes \vb - \frac{h^2(t)}{n}I_n\right)(\psi + \wtilde{\alpha})\,d\nu_t\\
& = \int_{\sphere_t^{n-1}}\left(\vb\otimes \vb - \frac{h^2(t)}{n}I_n\right)\psi\,d\nu_t + \wtilde{\alpha}  \int_{\sphere_t^{n-1}}\left(\vb\otimes \vb - \frac{h^2(t)}{n}I_n\right)\,d\nu_t \\	
& = \int_{\sphere_t^{n-1}}\left(\vb\otimes \vb - \frac{h^2(t)}{n}I_n\right)\psi\,d\nu_t\\
& = \T_t(\psi),
\end{align*}
which implies
\begin{align}\label{EEqP12}
\T_t(\psi) = \T_t(\phi)\in\wtilde{\K}_t^{co}.
\end{align}
In particular, for $\displaystyle\Vert \psi\Vert_{\Cb(\sphere_t^{n-1})} <\frac{1}{2}$, the inequality \eqref{EEqP11} holds, due to the following fact:
\begin{align*}
\wtilde{\alpha} = 1 - \int_{\sphere_t^{n-1}} \psi\, d\nu_t \geq 1- \Vert \psi\Vert_{\Cb(\sphere_t^{n-1})} > 1 -\frac{1}{2} = \frac{1}{2}> \Vert \psi\Vert_{\Cb(\sphere_t^{n-1})} \geq 0.
\end{align*}
Therefore, from equation \eqref{EEqP12}, we deduce
\begin{align}\label{EEqP13}
\T_t(\psi) \in\wtilde{\K}_t^{co}, \text{ for every } \psi\in \sphere_{1/2} \subset \Cb(\sphere_t^{n-1}),
\end{align}
where $\displaystyle \sphere_{1/2}:= \left\{f\in\Cb(\sphere_t^{n-1}) : \Vert f\Vert_{\Cb(\sphere_t^{n-1})} = \frac{1}{2} \right\}$.
\vskip 2mm
\noindent
\textbf{Step 5.} Let us now prove that the map $\T_t$ is a bounded.
Consider $\phi \in \Cb(\sphere_t^{n-1})$, then
\begin{align*}
\abs{\T_t(\phi)} & = \abs{ \int_{\sphere_t^{n-1}}\left(\vb, \vb\otimes \vb - \frac{h^2(t)}{n}I_n\right)\phi\,d\nu_t}  \leq  \int_{\sphere_t^{n-1}}\abs{\left(\vb, \vb\otimes \vb - \frac{h^2(t)}{n}I_n\right)}\abs{\phi}\,d\nu_t\\
& \leq  \Vert\phi\Vert_{\Cb(\sphere_t^{n-1})}\int_{\sphere_t^{n-1}}\abs{\left(\vb, \vb\otimes \vb - \frac{h^2(t)}{n}I_n\right)}\,d\nu_t,
\end{align*}
which implies
\begin{align*}
\|T\|_{op} = \sup_{0\neq \phi\in\Cb(\sphere_t^{n-1})} \frac{\abs{\T_t(\phi)}}{\Vert\phi\Vert_{\Cb(\sphere_t^{n-1})}} & \leq \int_{\sphere_t^{n-1}} \abs{\left(\vb, \vb\otimes \vb - \frac{h^2(t)}{n}I_n\right)} \,d\nu_t < \infty.
\end{align*}
Thus, for a.e. fixed $t\in [0,T]$, $\T_t$ is a bounded map. 

It suffices to show that $\T_t$ is surjective to prove that $\wtilde{\K}_t^{co}$ contains a neighborhood of $\0$.
Since $\T_t$ is linear bounded operator, we have $\T_t(0)=\coma{0}$. Moreover, if it is surjective also, then by the Open Mapping Theorem, $\T_t$ is an open map. Consequently, $\T_t$ maps a neighborhood of $0\in\Cb(\sphere_t^{n-1})$ onto a neighborhood of $\coma{0}\in\R^n\times M_0^n$. Precisely, if $\sphere_{1/2}\subset\Cb(\sphere_t^{n-1})$, then from \eqref{EEqP13}, it follows that
\begin{equation*}
T\big(\sphere_{1/2}\big) \subseteq \partial B_r(\coma{0}) \subset \wtilde{\K}_t^{co},\ \text{ for some }\ 0< r \leq 1.
\end{equation*} 
Thus, it follows that $\coma{0}\in\wtilde{\U}_t$.

\vskip 1mm
\noindent
\textbf{Step 6.} Finally, it remains to prove that the map $\T_t$ is surjective, which follows from an application of orthogonality in $\mL^2(\sphere_t^{n-1})$. 
\vskip 2mm
\noindent
\textbf{1.} Indeed, if for a fixed  $1\leq i\leq n$,  function $\phi$ is defined by
\[
\phi: \sphere^{n-1}  \ni \vb=(v_1, \cdots, v_n) \mapsto  v_i \in \mathbb{R},
\]
then 
\begin{align*}
\T_t(\phi) = \int_{\sphere_t^{n-1}}\left(\vb, \vb\otimes \vb - \frac{h^2(t)}{n}I_n\right)v_i\,d\nu_t.
\end{align*}
Using the symmetricity of the domain $\sphere_t^{n-1}$, for $1\le i,j,k \le n$, we infer
\begin{align*}
& \int_{\sphere_t^{n-1}}v_i\, d\nu_t  =\int_{\sphere_t^{n-1}} -v_i\, d\nu_t = 0 = \int_{\sphere_t^{n-1}}v_j^2 v_i\, d\nu_t  = \int_{\sphere_t^{n-1}}v_j^2 (-v_i)\, d\nu_t,\\  & \int_{\sphere_t^{n-1}}v_jv_kv_i\, d\nu_t = 0\ \text{ and }  \int_{\sphere_t^{n-1}}v_jv_i\, d\nu_t =  \int_{\sphere_t^{n-1}}v_jv_i\delta_{i,j}\, d\nu_t.
\end{align*}
It follows that
\begin{align*}
\T_t(v_i) =  \beta_t^1(e_i,\adda{0}), \ \text{ where }\ \beta_t^1  = \int_{\sphere_t^{n-1}} v_1^2 \, d\nu_t. 
\end{align*}
Hence it spans the Euclidean space $\R^n$.

\vskip 1mm
\noindent
\textbf{2.} Let us choose and fix $1\le  i, j, k, l \le n$ such that  $i\neq j$. 
Using the symmetry of $ \sphere_t^{n-1}$ similar to case 1, we obtain
\begin{align*} 
\int_{\sphere_t^{n-1}}v_kv_iv_j\, d\nu_t & = \int_{\sphere_t^{n-1}}(-v_k)v_i v_j\, d\nu_t = 0,
\\ 
\int_{\sphere_t^{n-1}}v_k^2v_iv_j\, d\nu_t & = \int_{\sphere_t^{n-1}}v_k^2(-v_i)v_j\, d\nu_t = \int_{\sphere_t^{n-1}}v_k^2v_i(-v_j)\, d\nu_t = 0,
\\
\int_{\sphere_t^{n-1}}v_kv_lv_iv_j\, d\nu_t & = \int_{\sphere_t^{n-1}}v_iv_jv_kv_l\delta_{k,i}\delta_{l,j}\, d\nu_t 
= \int_{\sphere_t^{n-1}}v_i^2v_j^2\, d\nu_t.
\end{align*}
Then, setting $\phi(\vb) = v_i v_j$, it follows that
\begin{align*}
\T_t(v_iv_j)  = \beta_t^2 \left(\0,  e_i \otimes e_j + e_j \otimes e_i\right), \text{ where }  \beta_t^2 = \int_{\sphere_t^{n-1}} v_1^2v_2^2 \, d\nu_t.
\end{align*}
\vskip 2mm
\noindent
\textbf{3.} Finally, for fixed $1\leq i\leq n$, let us consider $\phi(\vb) = v_i^2 - \frac{h^2(t)}{n}$. Then, we find
\begin{align*}
\T_t(\phi)  = \int_{\sphere_t^{n-1}}\left(\vb, \vb\otimes \vb - \frac{h^2(t)}{n}I_n\right)\left(v_i^2 - \frac{h^2(t)}{n}\right)\,d\nu_t.
\end{align*}
Again by the symmetricity of $\sphere_t^{n-1}$, see Remark \ref{Symmetricity}, for any $1\le k,l \le n$, we infer
\begin{align*}
\int_{\sphere_t^{n-1}}v_k\left(v_i^2 -\frac{h^2(t)}{n}\right)\, d\nu_t  = \int_{\sphere_t^{n-1}}(-v_k)\left(v_i^2 -\frac{h^2(t)}{n}\right)\, d\nu_t & = 0,\\
\int_{\sphere_t^{n-1}}v_kv_lv_i^2\, d\nu_t  = \int_{\sphere_t^{n-1}}(-v_k)v_lv_i^2\, d\nu_t = \int_{\sphere_t^{n-1}}v_k(-v_l)v_i^2\, d\nu_t & =0,
\end{align*}
and  $\displaystyle \int_{\sphere_t^{n-1}}v_kv_l\, d\nu_t  =0$  for all $k\ne l$.
Moreover, for any $1 \leq i,j \leq n$ with $i \ne j$, we have 
\begin{align*}
\frac{-1}{n-1}\int_{\sphere_t^{n-1}}\left(v_i^2 -\frac{h^2(t)}{n}\right)^2\, d\nu_t & =\int_{\sphere_t^{n-1}}\left(v_j^2 -\frac{h^2(t)}{n}\right)\left(v_i^2 -\frac{h^2(t)}{n}\right)\, d\nu_t .
\end{align*}
By using the above identities, we finally assert
\begin{equation*}
\T_t(\phi) =  \beta_t^3 \bigg(\0, e_i \otimes e_i - \frac{1}{(n-1)} \sum_{j\neq i} e_j \otimes e_j\bigg),\ \text{ where }\ \beta_t^3 =  \int_{\sphere_t^{n-1}} \left(v_i^2 -\frac{h^2(t)}{n}\right)^2 \, d\nu_t.
\end{equation*}
Therefore, the image set obtained in the above three cases, i.e., 
\[\Bigg\{\underbrace{(e_i,\coma{0})}_{spans\;  \R^n}, \underbrace{\left(\0,  e_i \otimes e_j + e_j \otimes e_i\right)}_{symmetricity}, \underbrace{\bigg(\0, e_i \otimes e_i - \frac{1}{(n-1)} \sum_{j\neq i} e_j \otimes e_j\bigg)}_{trace-zero} : 1\leq i,j \leq n \Bigg\}\]
forms a basis of $\R^n\times  M_0^n$ with the dimension $\displaystyle \left(n+\frac{n(n+1)}{2} -1\right)$.\\
Thus, for a.e. fixed $t\in[0,T]$, the image of $\T_t$ contains $\displaystyle \left(n+\frac{n(n+1)}{2} -1\right)$ linearly independent elements, which is a basis for $\R^n\times M_0^n$. Hence, $\T_t$ is surjective, that further implies  $\adda{0} \in \wtilde{\U}_t$.
\end{proof}

\appendix
\medskip\noindent
\textbf{Acknowledgments:} The authors would like to thank the Bernoulli Center at EPFL for its hospitality during the period when this work was discussed. 
A. Bawalia gratefully acknowledges the National Board of Higher Mathematics (NBHM) for travel support (Sanction No.:~0207/9(2)/2024-R$\&$D-II/10984), as well as the University Grants Commission (UGC), Government of India, for financial assistance (File No.:~368/2022/211610061684).
Z. Brze\'zniak and M.T. Mohan gratefully acknowledge the London Mathematical Society (Scheme 5, Grant Ref. 52426) for supporting M.T. Mohan’s visit to the University of York, during which part of this work was discussed. M.T. Mohan would  like to thank the Department of Science and Technology (DST) Science $\&$ Engineering Research Board (SERB), India for a MATRICS grant (MTR/2021/000066).  A. Bawalia would like to thank Dr.~K. Kinra, Prof.~K. Yamazaki, and Prof.~U. Koley  for helpful discussions.

\medskip\noindent	\textbf{Declarations:} 

\noindent 	\textbf{Ethical Approval:}   Not applicable 

\noindent  \textbf{Competing interests: } The authors declare no competing interests. 

\noindent  \textbf{Conflict of interest: }On behalf of all authors, the corresponding author states that there is no conflict of interest.

\noindent 	\textbf{Authors' contributions:} All authors have contributed equally. 


\noindent 	\textbf{Availability of data and materials:} Not applicable.

\bibliographystyle{plain}
\bibliography{Nonunique_ref}

\begin{thebibliography}{10}

\bibitem{LCB+EC+RS-25}
L.~C. Berselli, E.~Chiodaroli, and R.~Sannipoli.
\newblock Energy conservation for 3{D} {E}uler and {N}avier-{S}tokes equations
  in a bounded domain: applications to {B}eltrami flows.
\newblock {\em J. Nonlinear Sci.}, 35(1):Paper No. 10, 30, 2025.

\bibitem{DWB+SM+EST-24}
D.~W. Boutros, S.~Markfelder, and E.~S. Titi.
\newblock Nonuniqueness of generalised weak solutions to the primitive and
  {P}randtl equations.
\newblock {\em J. Nonlinear Sci.}, 34(4):Paper No. 68, 83, 2024.

\bibitem{DB+EF+MH-20}
D.~Breit, E.~Feireisl, and M.~Hofmanov\'{a}.
\newblock On solvability and ill-posedness of the compressible {E}uler system
  subject to stochastic forces.
\newblock {\em Anal. PDE}, 13(2):371--402, 2020.

\bibitem{PB+EW-25}
P.~Brkic and E.~Wiedemann.
\newblock Wild solutions of the three-dimensional axisymmetric {E}uler
  equations.
\newblock {\em SIAM J. Math. Anal.}, 57(1):996--1020, 2025.

\bibitem{CAB+CMFL+JHLN-15}
A.~C. Bronzi, M.~C. Lopes~Filho, and H.~J. Nussenzveig~Lopes.
\newblock Wild solutions for 2{D} incompressible ideal flow with passive
  tracer.
\newblock {\em Commun. Math. Sci.}, 13(5):1333--1343, 2015.

\bibitem{ZB+FF+MM-16}
Z.~Brze\'zniak, F.~Flandoli, and M.~Maurelli.
\newblock Existence and uniqueness for stochastic 2{D} {E}uler flows with
  bounded vorticity.
\newblock {\em Arch. Ration. Mech. Anal.}, 221(1):107--142, 2016.

\bibitem{ZB+UM+DM-19}
Z.~Brze\'zniak, U.~Manna, and D.~Mukherjee.
\newblock Wong-{Z}akai approximation for the stochastic
  {L}andau-{L}ifshitz-{G}ilbert equations.
\newblock {\em J. Differential Equations}, 267(2):776--825, 2019.

\bibitem{ZB+MM-20}
Z.~Brze\'zniak and M.~Maurelli.
\newblock Existence for stochastic 2{D} {E}uler equations with positive
  ${H}^{-1}$ vorticity.
\newblock {\em \href{https://doi.org/10.1007/s40072-026-00422-2}{To appear in
  Stoch. Partial Differ. Equ. Anal. Comput.}}, 2026.

\bibitem{ZB+PS-01}
Z.~Brze\'zniak and S.~Peszat.
\newblock Stochastic two dimensional {E}uler equations.
\newblock {\em Ann. Probab.}, 29(4):1796--1832, 2001.

\bibitem{TB-15}
T.~Buckmaster.
\newblock Onsager's conjecture almost everywhere in time.
\newblock {\em Comm. Math. Phys.}, 333(3):1175--1198, 2015.

\bibitem{TB+MC+VV-22}
T.~Buckmaster, M.~Colombo, and V.~Vicol.
\newblock Wild solutions of the {N}avier-{S}tokes equations whose singular sets
  in time have {H}ausdorff dimension strictly less than 1.
\newblock {\em J. Eur. Math. Soc. (JEMS)}, 24(9):3333--3378, 2022.

\bibitem{TB+CDL+PI+LSJ-15}
T.~Buckmaster, C.~De~Lellis, P.~Isett, and L.~Sz\'{e}kelyhidi, Jr.
\newblock Anomalous dissipation for {$1/5$}-{H}\"{o}lder {E}uler flows.
\newblock {\em Ann. of Math. (2)}, 182(1):127--172, 2015.

\bibitem{TB+CDL+LSJ-16}
T.~Buckmaster, C.~De~Lellis, and L.~Sz\'{e}kelyhidi, Jr.
\newblock Dissipative {E}uler flows with {O}nsager-critical spatial regularity.
\newblock {\em Comm. Pure Appl. Math.}, 69(9):1613--1670, 2016.

\bibitem{TB+CDL+LSJ+VV-19}
T.~Buckmaster, C.~De~Lellis, L.~Sz\'{e}kelyhidi, Jr., and V.~Vicol.
\newblock Onsager's conjecture for admissible weak solutions.
\newblock {\em Comm. Pure Appl. Math.}, 72(2):229--274, 2019.

\bibitem{TB+VV-19}
T.~Buckmaster and V.~Vicol.
\newblock Convex integration and phenomenologies in turbulence.
\newblock {\em EMS Surv. Math. Sci.}, 6(1-2):173--263, 2019.

\bibitem{TB+VV-19-NSE}
T.~Buckmaster and V.~Vicol.
\newblock Nonuniqueness of weak solutions to the {N}avier-{S}tokes equation.
\newblock {\em Ann. of Math. (2)}, 189(1):101--144, 2019.

\bibitem{MC+ZZ-25}
M.~Changxing and Z.~Zhiwen.
\newblock Nonuniqueness for high-dimensional ideal {M}{H}{D} equations via
  differential inclusion.
\newblock {\em \href{https://arxiv.org/pdf/2509.06866}{arXiv:2509.06866}},
  2025.

\bibitem{WC+ZD+XZ-24}
W.~Chen, Z.~Dong, and X.~Zhu.
\newblock Sharp nonuniqueness of solutions to stochastic {N}avier-{S}tokes
  equations.
\newblock {\em SIAM J. Math. Anal.}, 56(2):2248--2285, 2024.

\bibitem{AC+PC+SF+RS-08}
A.~Cheskidov, P.~Constantin, S.~Friedlander, and R.~Shvydkoy.
\newblock Energy conservation and {O}nsager's conjecture for the {E}uler
  equations.
\newblock {\em Nonlinearity}, 21(6):1233--1252, 2008.

\bibitem{AC+XL-22}
A.~Cheskidov and X.~Luo.
\newblock Sharp nonuniqueness for the {N}avier-{S}tokes equations.
\newblock {\em Invent. Math.}, 229(3):987--1054, 2022.

\bibitem{EC-14}
E.~Chiodaroli.
\newblock A counterexample to well-posedness of entropy solutions to the
  compressible {E}uler system.
\newblock {\em J. Hyperbolic Differ. Equ.}, 11(3):493--519, 2014.

\bibitem{EC+CDL+OK-15}
E.~Chiodaroli, C.~De~Lellis, and O.~Kreml.
\newblock Global ill-posedness of the isentropic system of gas dynamics.
\newblock {\em Comm. Pure Appl. Math.}, 68(7):1157--1190, 2015.

\bibitem{EC+EF+FF-21}
E.~Chiodaroli, E.~Feireisl, and F.~Flandoli.
\newblock Ill-posedness for the full {E}uler system driven by multiplicative
  white noise.
\newblock {\em Indiana Univ. Math. J.}, 70(4):1267--1282, 2021.

\bibitem{EC+OK-14}
E.~Chiodaroli and O.~Kreml.
\newblock On the energy dissipation rate of solutions to the compressible
  isentropic {E}uler system.
\newblock {\em Arch. Ration. Mech. Anal.}, 214(3):1019--1049, 2014.

\bibitem{AC+CDL+LSJ-12}
A.~Choffrut, C.~De~Lellis, and L.~Sz{\'e}kelyhidi~Jr.
\newblock Dissipative continuous {E}uler flows in two and three dimensions.
\newblock {\em \href{https://arxiv.org/pdf/1205.1226v1}{arXiv:1205.1226v1}},
  2012.

\bibitem{AC+LSJ-14}
A.~Choffrut and L.~Sz\'{e}kelyhidi, Jr.
\newblock Weak solutions to the stationary incompressible {E}uler equations.
\newblock {\em SIAM J. Math. Anal.}, 46(6):4060--4074, 2014.

\bibitem{MC+CDL+LDR-18}
M.~Colombo, C.~De~Lellis, and L.~De~Rosa.
\newblock Ill-posedness of {L}eray solutions for the hypodissipative
  {N}avier-{S}tokes equations.
\newblock {\em Comm. Math. Phys.}, 362(2):659--688, 2018.

\bibitem{MC+LDR+MS-22}
M.~Colombo, L.~De~Rosa, and M.~Sorella.
\newblock Typicality results for weak solutions of the incompressible
  {N}avier-{S}tokes equations.
\newblock {\em ESAIM Control Optim. Calc. Var.}, 28:Paper No. 38, 24, 2022.

\bibitem{PC+WE+ET-91}
P.~Constantin, W.~E, and E.~S. Titi.
\newblock Onsager's conjecture on the energy conservation for solutions of
  {E}uler's equation.
\newblock {\em Comm. Math. Phys.}, 165(1):207--209, 1994.

\bibitem{PC+AM-88}
P.~Constantin and A.~Majda.
\newblock The {B}eltrami spectrum for incompressible fluid flows.
\newblock {\em Comm. Math. Phys.}, 115(3):435--456, 1988.

\bibitem{DC+DF+FG-11}
D.~Cordoba, D.~Faraco, and F.~Gancedo.
\newblock Lack of uniqueness for weak solutions of the incompressible porous
  media equation.
\newblock {\em Arch. Ration. Mech. Anal.}, 200(3):725--746, 2011.

\bibitem{SD-14}
S.~Daneri.
\newblock Cauchy problem for dissipative {H}\"{o}lder solutions to the
  incompressible {E}uler equations.
\newblock {\em Comm. Math. Phys.}, 329(2):745--786, 2014.

\bibitem{SD+LSJ-17}
S.~Daneri and L.~Sz\'{e}kelyhidi, Jr.
\newblock Non-uniqueness and h-principle for {H}\"{o}lder-continuous weak
  solutions of the {E}uler equations.
\newblock {\em Arch. Ration. Mech. Anal.}, 224(2):471--514, 2017.

\bibitem{CDL-08}
C.~De~Lellis.
\newblock Ill-posedness for bounded admissible solutions of the 2-dimensional
  {$p$}-system.
\newblock In {\em Hyperbolic problems: theory, numerics and applications},
  volume~67, pages 269--278. Amer. Math. Soc., Providence, RI, 2009.

\bibitem{CDL+EC+OK-14}
C.~De~Lellis, E.~Chiodaroli, and O.~Kreml.
\newblock Surprising solutions to the isentropic {E}uler system of gas
  dynamics.
\newblock In {\em Hyperbolic problems: theory, numerics, applications},
  volume~8, pages 1--10. Am. Inst. Math. Sci. (AIMS), Springfield, MO, 2014.

\bibitem{CDL+LSJ-09}
C.~De~Lellis and L.~Sz\'{e}kelyhidi, Jr.
\newblock The {E}uler equations as a differential inclusion.
\newblock {\em Ann. of Math. (2)}, 170(3):1417--1436, 2009.

\bibitem{CDL+LSJ-10}
C.~De~Lellis and L.~Sz\'{e}kelyhidi, Jr.
\newblock On admissibility criteria for weak solutions of the {E}uler
  equations.
\newblock {\em Arch. Ration. Mech. Anal.}, 195(1):225--260, 2010.

\bibitem{CDL+LSJ-13}
C.~De~Lellis and L.~Sz\'{e}kelyhidi, Jr.
\newblock Dissipative continuous {E}uler flows.
\newblock {\em Invent. Math.}, 193(2):377--407, 2013.

\bibitem{CDL+LSJ-14}
C.~De~Lellis and L.~Sz\'{e}kelyhidi, Jr.
\newblock Dissipative {E}uler flows and {O}nsager's conjecture.
\newblock {\em J. Eur. Math. Soc. (JEMS)}, 16(7):1467--1505, 2014.

\bibitem{HD-77}
H.~Doss.
\newblock Liens entre \'equations diff\'erentielles stochastiques et
  ordinaires.
\newblock {\em Ann. Inst. H. Poincar\'e{} Sect. B (N.S.)}, 13(2):99--125, 1977.

\bibitem{LE-57}
L.~Euler.
\newblock {P}rincipes g\'en\'eraux du mouvement des fluides.
\newblock 11:274--315, 1757.

\bibitem{GLE-94}
G.~L. Eyink.
\newblock Energy dissipation without viscosity in ideal hydrodynamics. {I}.
  {F}ourier analysis and local energy transfer.
\newblock {\em Phys. D}, 78(3-4):222--240, 1994.

\bibitem{CF-06}
C.~L. Fefferman.
\newblock Existence and smoothness of the {N}avier-{S}tokes equation.
\newblock In {\em The millennium prize problems}, pages 57--67. Clay Math.
  Inst., Cambridge, MA, 2006.

\bibitem{EF-14}
E.~Feireisl.
\newblock Maximal dissipation and well-posedness for the compressible {E}uler
  system.
\newblock {\em J. Math. Fluid Mech.}, 16(1):447--461, 2014.

\bibitem{FF+MG+EP-10}
F.~Flandoli, M.~Gubinelli, and E.~Priola.
\newblock Well-posedness of the transport equation by stochastic perturbation.
\newblock {\em Invent. Math.}, 180(1):1--53, 2010.

\bibitem{VG+HK+MN-24}
V.~Giri, H.~Kwon, and M.~Novack.
\newblock A wavelet-inspired {$L^3$}-based convex integration framework for the
  {E}uler equations.
\newblock {\em Ann. PDE}, 10(2):Paper No. 19, 271, 2024.

\bibitem{VG+ROR-24}
V.~Giri and R.~O. Radu.
\newblock The {O}nsager conjecture in 2{D}: a {N}ewton-{N}ash iteration.
\newblock {\em Invent. Math.}, 238(2):691--768, 2024.

\bibitem{HG+EN+CVV-14}
N.~E. Glatt-Holtz and V.~C. Vicol.
\newblock Local and global existence of smooth solutions for the stochastic
  {E}uler equations with multiplicative noise.
\newblock {\em Ann. Probab.}, 42(1):80--145, 2014.

\bibitem{MH+TL+UP-24}
M.~Hofmanov\'{a}, T.~Lange, and U.~Pappalettera.
\newblock Global existence and non-uniqueness of 3{D} {E}uler equations
  perturbed by transport noise.
\newblock {\em Probab. Theory Related Fields}, 188(3-4):1183--1255, 2024.

\bibitem{MH+RZ+XZ-22}
M.~Hofmanov\'{a}, R.~Zhu, and X.~Zhu.
\newblock On ill- and well-posedness of dissipative martingale solutions to
  stochastic 3{D} {E}uler equations.
\newblock {\em Comm. Pure Appl. Math.}, 75(11):2446--2510, 2022.

\bibitem{MH+RZ+XZ-23}
M.~Hofmanov\'{a}, R.~Zhu, and X.~Zhu.
\newblock Global existence and non-uniqueness for 3{D} {N}avier-{S}tokes
  equations with space-time white noise.
\newblock {\em Arch. Ration. Mech. Anal.}, 247(3):Paper No. 46, 70, 2023.

\bibitem{MH+RZ+XZ-23-trace-class}
M.~Hofmanov\'{a}, R.~Zhu, and X.~Zhu.
\newblock Global-in-time probabilistically strong and {M}arkov solutions to
  stochastic 3{D} {N}avier-{S}tokes equations: existence and nonuniqueness.
\newblock {\em Ann. Probab.}, 51(2):524--579, 2023.

\bibitem{MH+RZ+XZ-24}
M.~Hofmanov\'{a}, R.~Zhu, and X.~Zhu.
\newblock Nonuniqueness in law of stochastic 3{D} {N}avier-{S}tokes equations.
\newblock {\em J. Eur. Math. Soc. (JEMS)}, 26(1):163--260, 2024.

\bibitem{MH+RZ+XZ-25}
M.~Hofmanov\'a, R.~Zhu, and X.~Zhu.
\newblock Non-unique ergodicity for deterministic and stochastic 3{D}
  {N}avier-{S}tokes and {E}uler equations.
\newblock {\em Arch. Ration. Mech. Anal.}, 249(3):Paper No. 33, 2025.

\bibitem{PI-12}
P.~Isett.
\newblock {\em H\"{o}lder continuous {E}uler flows with compact support in
  time}.
\newblock ProQuest LLC, Ann Arbor, MI, 2013.

\bibitem{PI-18}
P.~Isett.
\newblock A proof of {O}nsager's conjecture.
\newblock {\em Ann. of Math. (2)}, 188(3):871--963, 2018.

\bibitem{PS-24}
P.~Isett.
\newblock On the endpoint regularity in {O}nsager's conjecture.
\newblock {\em Anal. PDE}, 17(6):2123--2159, 2024.

\bibitem{PI+SJO-16}
P.~Isett and S.~J. Oh.
\newblock A heat flow approach to {O}nsager's conjecture for the {E}uler
  equations on manifolds.
\newblock {\em Trans. Amer. Math. Soc.}, 368(9):6519--6537, 2016.

\bibitem{PI+VV-15}
P.~Isett and V.~Vicol.
\newblock H\"{o}lder continuous solutions of active scalar equations.
\newblock {\em Ann. PDE}, 1(1):Art. 2, 77, 2015.

\bibitem{UK+KY-25}
U.~Koley and K.~Yamazaki.
\newblock Non-uniqueness in law of transport-diffusion equation forced by
  random noise.
\newblock {\em J. Differential Equations}, 416:82--142, 2025.

\bibitem{NHK-55}
N.~H. Kuiper.
\newblock On {$C^1$}-isometric imbeddings. {I}, {II}.
\newblock {\em Indag. Math.}, 17:545--556, 683--689, 1955.

\bibitem{HK-90}
H.~Kunita.
\newblock {\em Stochastic flows and stochastic differential equations},
  volume~24.
\newblock Cambridge University Press, Cambridge, 1990.

\bibitem{SL-88}
S.~\L{}ojasiewicz.
\newblock {\em An introduction to the theory of real functions}.
\newblock John Wiley \& Sons, Ltd., Chichester, third edition, 1988.

\bibitem{HL+XZ-25}
H.~L\"u and X.~Zhu.
\newblock Non-unique ergodicity for the 2{D} stochastic {N}avier-{S}tokes
  equations with derivative of space-time white noise.
\newblock {\em J. Differential Equations}, 425:383--433, 2025.

\bibitem{HL+XZ-24}
H.~L\"u and X.~Zhu.
\newblock Sharp nonuniqueness of solutions to 2{D} {N}avier-{S}tokes equations
  with space-time white noise.
\newblock {\em Ann. Appl. Probab.}, 35(3):1980--2030, 2025.

\bibitem{Lu+Lu+Zhu-25+}
H.~Lü, L.~Lü, and R.~Zhu.
\newblock A proof of {O}nsager's conjecture for the stochastic 3{D} {E}uler
  equations.
\newblock {\em \href{https://arxiv.org/abs/2505.06915}{arXiv:2505.06915}},
  2025.

\bibitem{SM+PQ-24}
S.~Mao and P.~Qu.
\newblock Non-uniqueness for the compressible {E}uler-{M}axwell equations.
\newblock {\em Calc. Var. Partial Differential Equations}, 63(7):Paper No. 186,
  84, 2024.

\bibitem{FM-24}
F.~Mengual.
\newblock Non-uniqueness of admissible solutions for the 2{D} {E}uler equation
  with {$L^p$} vortex data.
\newblock {\em Comm. Math. Phys.}, 405(9):Paper No. 207, 28, 2024.

\bibitem{SM+LSJ-18}
S.~Modena and L.~Sz\'ekelyhidi, Jr.
\newblock Non-uniqueness for the transport equation with {S}obolev vector
  fields.
\newblock {\em Ann. PDE}, 4(2):Paper No. 18, 38, 2018.

\bibitem{JN-54}
J.~Nash.
\newblock {$C^1$} isometric imbeddings.
\newblock {\em Ann. of Math. (2)}, 60:383--396, 1954.

\bibitem{LO-49}
L.~Onsager.
\newblock Statistical hydodynamics.
\newblock {\em Convegno Internazionale di Meccanica Statistica}, 2(6):279--287,
  1949.

\bibitem{JCO-80}
J.~C. Oxtoby.
\newblock {\em Measure and category}, volume~2.
\newblock Springer-Verlag, New York-Berlin, second edition, 1980.

\bibitem{RTR-70}
R.~T. Rockafellar.
\newblock {\em Convex analysis}.
\newblock Princeton University Press, Princeton, NJ, 1970.

\bibitem{WR-73}
W.~Rudin.
\newblock {\em Functional analysis}.
\newblock McGraw-Hill Book Co., New York-D\"usseldorf-Johannesburg, 1973.

\bibitem{VS-93}
V.~Scheffer.
\newblock An inviscid flow with compact support in space-time.
\newblock {\em J. Geom. Anal.}, 3(4):343--401, 1993.

\bibitem{AS-97}
A.~Shnirelman.
\newblock On the nonuniqueness of weak solution of the {E}uler equation.
\newblock {\em Comm. Pure Appl. Math.}, 50(12):1261--1286, 1997.

\bibitem{RS-11}
R.~Shvydkoy.
\newblock Convex integration for a class of active scalar equations.
\newblock {\em J. Amer. Math. Soc.}, 24(4):1159--1174, 2011.

\bibitem{JHS-78}
H.~J. Sussmann.
\newblock On the gap between deterministic and stochastic ordinary differential
  equations.
\newblock {\em Ann. Probability}, 6(1):19--41, 1978.

\bibitem{LSJ-11}
L.~Sz\'{e}kelyhidi.
\newblock Weak solutions to the incompressible {E}uler equations with vortex
  sheet initial data.
\newblock {\em C. R. Math. Acad. Sci. Paris}, 349(19-20):1063--1066, 2011.

\bibitem{LSJ-12}
L.~Sz\'{e}kelyhidi, Jr.
\newblock Relaxation of the incompressible porous media equation.
\newblock {\em Ann. Sci. \'{E}c. Norm. Sup\'{e}r. (4)}, 45(3):491--509, 2012.

\bibitem{LSJ-13}
L.~Sz\'{e}kelyhidi, Jr.
\newblock From isometric embeddings to turbulence.
\newblock In {\em H{CDTE} lecture notes. {P}art {II}. {N}onlinear hyperbolic
  {PDE}s, dispersive and transport equations}, volume~7, page~63. Am. Inst.
  Math. Sci. (AIMS), Springfield, MO, 2013.

\bibitem{LT-79}
L.~Tartar.
\newblock Compensated compactness and applications to partial differential
  equations.
\newblock In {\em Nonlinear analysis and mechanics: {H}eriot-{W}att
  {S}ymposium, {V}ol. {IV}}, volume~39, pages 136--212. Pitman, Boston,
  Mass.-London, 1979.

\bibitem{RT-01}
R.~Temam.
\newblock {\em Navier-{S}tokes equations}.
\newblock AMS Chelsea Publishing, Providence, RI, 2001.

\bibitem{EW-11}
E.~Wiedemann.
\newblock Existence of weak solutions for the incompressible {E}uler equations.
\newblock {\em Ann. Inst. H. Poincar\'{e} C Anal. Non Lin\'{e}aire},
  28(5):727--730, 2011.

\bibitem{KY-22}
K.~Yamazaki.
\newblock Nonuniqueness in law for two-dimensional {N}avier--{S}tokes equations
  with diffusion weaker than a full {L}aplacian.
\newblock {\em SIAM Journal on Mathematical Analysis}, 54(4):3997--4042, 2022.

\bibitem{KY-22-remarks}
K.~Yamazaki.
\newblock Remarks on the non-uniqueness in law of the {N}avier–{S}tokes
  equations up to the {J}.-{L}. {L}ions’ exponent.
\newblock {\em Stochastic Processes and their Applications}, 147(4):226--269,
  2022.

\bibitem{KY-24}
K.~Yamazaki.
\newblock Non-uniqueness in law of three-dimensional {N}avier-{S}tokes
  equations diffused via a fractional {L}aplacian with power less than one
  half.
\newblock {\em Stoch. Partial Differ. Equ. Anal. Comput.}, 12(1):794--855,
  2024.

\bibitem{KY-26}
K.~Yamazaki.
\newblock Remarks on the two-dimensional magnetohydrodynamics system forced by
  space-time white noise.
\newblock {\em Stochastic Process. Appl.}, 195:Paper No. 104893, 38, 2026.

\bibitem{WPZ-12}
W.~P. Ziemer.
\newblock {\em Weakly differentiable functions: Sobolev spaces and functions of
  bounded variation}, volume 120.
\newblock Springer Science \& Business Media, 2012.

\end{thebibliography}

\end{document}